\documentclass[psamsfonts]{amsart}
\usepackage{
graphicx,
fancyhdr,
amsmath,
amsfonts,
dsfont,
amsthm,
amssymb,
enumerate,
mathrsfs,
hyperref,
xcolor,
ulem,
cancel,
cleveref
}

\newcommand{\epf}{{{\hfill $\Box$ \smallskip}}}

\newcommand{\E}{\mathbb{E}}
\newcommand{\R}{\mathbb{R}}

\newcommand{\Z}{\mathbb{Z}}

\newcommand{\N}{\mathbb{N}}

\newcommand{\1}{\mathds{1}}

\newcommand{\Prb}{\mathbb{P}}

\newcommand{\varphiLaw}{\mathsf{Q}}
\newcommand{\Fprime}{\Theta}

\newtheorem{proposition}{Proposition}[section]
\newtheorem{corollary}{Corollary}[section]
\newtheorem{lemma}{Lemma}[section]
\newtheorem{theorem}{Theorem}[section]

\theoremstyle{definition}

\numberwithin{equation}{section}

\makeatletter
\@namedef{subjclassname@2020}{\textup{2020} Mathematics Subject Classification}
\makeatother

\title[Differentiability of the Effective Lagrangian]{Differentiability of the effective Lagrangian for Hamilton--Jacobi--Bellman equations in dynamic random environments
}

\author{Yuri Bakhtin}
\address{Courant Institute of Mathematical Sciences, New York University, 251 Mercer St, New York, NY 10012, USA}
\thanks{The first author is grateful to NSF for partial support via grant DMS-1811444.}

\email{bakhtin@cims.nyu.edu}

\author{Douglas Dow}
\address{Courant Institute of Mathematical Sciences, New York University, 251 Mercer St, New York, NY 10012, USA}
\email{dd3103@cims.nyu.edu}


\begin{document}

\maketitle
\begin{abstract}
 We prove differentiability of the effective Lagrangian for continuous time multidimensional directed variational problems in random dynamic environments 
 with
positive  dependence range in time. This implies that limiting fundamental solutions in the associated homogenization problems for HJB equations are classical.
\end{abstract}


%

\section{Introduction}
\subsection{Main results.}
In this paper, we consider random Lagrangian actions defined on paths $\gamma$ in dynamic random environments in $1+d$ dimensions ($d\in \N$) by
\begin{equation}
\label{eq:continuous-action}
A^t(\gamma)=\int_0^t L(\dot{\gamma}_s)ds + \int_0^t F(s,\gamma_s)ds,\quad \gamma:[0,t]\to\R^d.
\end{equation}
Here  $L:\R^d\to\R$ is a generalized kinetic energy function and $F:\R\times\R^d\to\R$ is a realization of a space-time stationary random potential.
We define the optimal action over paths connecting the origin at time $0$ and a point $x$ at time $t$ by 
\begin{equation}
\label{eq:optimal-action}
A(t,x)=\inf_{\gamma:\gamma(0)=0,\gamma(t)=x}A^t(\gamma).
\end{equation}

The main goal of this paper is to prove that  for
a large class of optimal Lagrangian actions defined by \eqref{eq:continuous-action} and~\eqref{eq:optimal-action},
the effective Lagrangian $\Lambda:\R^d\to\R$ defined by
\[
\Lambda(v)=\lim_{T\to\infty}\frac{1}{T}A(T,Tv),\quad v\in\R^d,
\]
is nonrandom, convex, and \textbf{differentiable} at all $v\in\R^d$.
The precise assumptions and rigorous statements (Theorems~\ref{thm:shapeTheorem} and~\ref{thm:mainDiffTheorem}) are given in Section~\ref{sec:setup}. We also provide a formula for
$\nabla \Lambda(v)$ valid for all $v\in\R^d$.

Under the assumptions we impose, $A(t,x)$ can be viewed as the fundamental viscosity solution of the associated Hamilton--Jacobi--Bellman (HJB) equation
with Hamiltonian given by the Legendre--Fenchel transform of~$L$. 
Our main result can be stated as a homogenization theorem with an additional regularity property: the rescaled HJB solution $\epsilon A(t/\epsilon,x/\epsilon)$ converges to a limit which can be viewed as the
fundamental solution
of an effective HJB equation; moreover, the limiting fundamental
solution is $C^1$, i.e., {\bf classical}. Equivalently, in the homogenized limit the effective Hamiltonian is strictly convex.  The precise statement (Theorem~\ref{th:main-HJB}) is also given in Section~\ref{sec:setup}.

\subsection{Background.}
For various random environment models, the minimal cost, or action $A(x)$ of a path connecting the origin to a point $x$ in a Euclidean space grows asymptotically linearly in $|x|$, with rate of growth
depending on the direction or slope~$v$. In other words, the limit 
\[
\lim_{T\to\infty}\frac{1}{T}A(Tv)\to \Lambda(v)
\]
exists and is finite $\Prb$-almost surely. Results of this kind may be  interpreted in terms of random growth or front propagation in random media.

The existence of a deterministic limit $\Lambda$ and its convexity are consequences of the subadditive ergodic theorem.
For First Passage Percolation (FPP) and Last Passage Percolation (LPP) models, see, e.g.,  \cite{AuffingerDamronHanson_50Years:MR3729447},\cite{Damron:MR3838442}, \cite{JRAS-JEMS}, and references therein.
The limiting function~$\Lambda$ is known as the {\it shape function} or, in the context of stochastic HJB equations, the {\it effective Lagrangian}.
In the latter setting, such limits are tightly related to the homogenization phenomenon. For the case of a dynamic random environment, they were established in 
\cite{Schwab:MR2514380}, \cite{BCK:MR3110798},
\cite{kickb:bakhtin2016}, \cite{Seeger:MR4266240}. Similar results
for HJB equations with positive viscosity have been obtained in  
\cite{Kosygina-Varadhan:MR2400607},     
\cite{Jing-Souganidis-Tran:MR3602941},
\cite{Bakhtin-Li:MR3911894},
 \cite{JRAS:https://doi.org/10.48550/arxiv.2211.06779}.

Shape functions have been computed (in some cases only up to an unknown parameter) for a small number of models allowing for a group of precise 
distributional symmetries or even exact solvability, see \cite{Hammersley:MR0405665},  \cite{Rost:MR635270}, \cite{Aldous-Diaconis:MR1355056}, 
\cite{HoNe3},
\cite{Baryshnikov:MR1818248},\cite{Gravner-Tracy-Widom:MR1830441}, \cite{Hambly-Martin-O'Connell:MR1935124}, \cite{Moriarty-O'Connell:MR2343849},
\cite{CaPi}, \cite{Seppalainen:MR2917766}, \cite{BCK:MR3110798},
\cite{kickb:bakhtin2016}, \cite{Bakhtin-Li:MR3911894}, \cite{JRAS:https://doi.org/10.48550/arxiv.2211.06779}, and a discussion in~\cite{BakhtinDow_Differentiability}. In particular, the quadratic kinetic energy
models considered in \cite{BCK:MR3110798}, \cite{kickb:bakhtin2016}, \cite{Bakhtin-Li:MR3911894}, and \cite{JRAS:https://doi.org/10.48550/arxiv.2211.06779} have a shear invariance property resulting in precisely quadratic effective Lagrangians.

A long-standing conjecture supported by these precise limits is that for a broad family of 
models essentially coinciding with the KPZ universality class, the shape functions are strictly convex and differentiable, i.e., they have no flat edges and no corners,
 see, e.g., \cite{BK18}.

Besides the handful of aforementioned explicit examples, there has been no progress on this conjecture until~\cite{BakhtinDow_Differentiability}, where differentiability of $\Lambda$ was shown for a model 
where the cost of a directed discrete time $\R^1$-valued path~$\gamma=(\gamma_0, \gamma_1,\dots,\gamma_n)$
is given
by
\begin{equation}
\label{eq:time-discrete-action}
A^n(\gamma)=\sum_{k=0}^{n-1} L(\gamma_{k+1}-\gamma_{k})+\sum_{k=0}^{n-1}F_k(\gamma_k)
\end{equation}
under broad assumptions on smooth nearest neighbor interaction energy $L$ and the requirement that 
the random potential $F$ is i.i.d.\ in time (in other words, $F$~is of zero dependence range, or {\it white} in time), smooth and stationary in space, and bounded from below.
A similar result was shown in~\cite{BakhtinDow_Differentiability} to hold for rates of growth of the free energy for 
positive temperature point-to-point 
directed polymers. For general~$L$, the model based on the time-discrete action \eqref{eq:time-discrete-action} is not shear-invariant but the white-in-time potential is, and this allows for estimates based on small
shear perturbations  of the environment.

\medskip

In this paper, we extend the shear method of \cite{BakhtinDow_Differentiability} in several directions and show that differentiability of the shape function holds in a broader setting. 
To start with, here we work in continuous time and in dimension $1+d$ for arbitrary $d\in \N$. We consider
variational problems for the action functional defined on continuous time paths by~\eqref{eq:continuous-action},
 impose certain nonrestrictive requirements --- twice differentiability, convexity, and a mild restriction on the character of growth at infinity --- on the generalized kinetic energy $L$,
and assume that the random environment is given by a space-time stationary random potential $F:\R\times \R^d\to\R$ with $C^2$ realizations.
We describe a class of random fields $F$ that allows for arbitrary values, not necessarily bounded from below. Most importantly, all these random fields 
have positive dependence range, so they
are not white in time. Thus, in contrast with the model studied in \cite{BakhtinDow_Differentiability}, shear invariance does not hold
even for the potential~$F$ in the model we consider in the present paper.
 The precise description of the setting and rigorous statements of our results are given in Section~\ref{sec:setup}.

Although our exposition is restricted to the setting described in Section~\ref{sec:setup},
our method also applies to various interpolations between that setting and the setting of~\cite{BakhtinDow_Differentiability}. For example, 
in the discrete time case, we also could consider non-i.i.d.\ potentials similar to those we consider here in continuous time. Or we could consider zero dependence
range models in continuous time, say, given by compound Poisson noise, etc. It is also clear that for various other continuous space models of
random growth and front propagation, 
once the existence of a deterministic convex shape function is known, its
differentiability can be established with our method.

The advantage of the setting adopted in the present paper is direct applicability of our results to homogenization for HJB equations, see Theorem~\ref{th:main-HJB}.

{\bf Acknowledgments.} We thank
Scott Armstrong,  Elena Kosygina, Lenya Ryzhik, Panagiotis Souganidis, and Hung~Tran for their comments.

\section{Set-up}\label{sec:setup}

In this section, we introduce the model in detail and give rigorous statements of the results.

\subsection{Notational conventions.} 
Throughout, we fix the dimension $d\in \N$.
 When referring to the space $\R\times \R^d$, we will often use $t$, $s$, or $T$ to represent the first coordinate, which we often refer to as the time coordinate, and $x$ or $y$ to represent the last $d$ coordinates, which we often refer to as the spatial coordinates. We use $\langle x,y\rangle$ to denote the Euclidean inner product between two vectors $x,y\in \R^d.$ We denote by $\|x\|_p,$ $p\in [1,\infty]$, the $\ell^p$ norm of a vector $x\in \R^d.$ We denote by $\|M\|$ the operator norm of a matrix $M$. For a function $f$ whose domain $U$ is a subset of Euclidean space we use $\|f\|_{L^p}$ or $\|f\|_{L^p(U)}$ to denote the~$L^p$ norm of either the functions $|f|$, $\|f\|_2$ or $\|f\|$, depending on whether the range of $f$ is $\R$, $\R^d$ or $\R^{d\times d}.$ For a function $f:\R\times \R^d\to \R$ we denote by $\nabla f$ and $\nabla^2 f$ the gradient and Hessian of $f$ with respect to the spatial variable $x\in \R^d.$ We use $\partial_t f$ to denote the time derivative of such $f$. For a differentiable path $\gamma:[t,T]\to \R^d$ we denote by $\dot{\gamma}$ its time derivative.

\subsection{The variational problem}
We will work with a convex function $L:\R^d\to \R$ serving as a generalized kinetic energy and a space-time stationary random field $F:\R\times \R^d\to \R$ with $C^2$
realizations, which serves as the environment potential. We will discuss requirements on $F$ and $L$ in Sections
\ref{sec:F} and \ref{sec:reqs_on_L}.

The  set  of admissible paths from point $x\in\R^d$ at time $t\in\R$ to point $y\in\R^d$ at time $T>t$ is defined as
\[\Gamma_{x,y}^{t,T} = \{\gamma\in W^{1,1}([t,T];\R^d)\,:\,\gamma_t = x,\,\gamma_T = y\},\] 
where $W^{1,1}([t,T];\R^d)$ is the Sobolev space of measurable functions $\gamma:[t,T]\to \R^d$ with weak derivative $\dot\gamma$ in $L^1([t,T];\R^d).$ The derivative $\dot\gamma$ is then defined pointwise Lebesgue almost everywhere. For $\gamma\in W^{1,1}([t,T];\R^d),$ let
\begin{equation}\label{eq:AtTDef}
    A^{t,T}(\gamma) = \int_t^T L(\dot{\gamma}_s)ds + \int_t^T F(s,\gamma_s)ds,
\end{equation}
which is a slight generalization of~\eqref{eq:continuous-action}.
We are interested in minimal actions 
\begin{equation}\label{eq:p2pAction}
    A^{t,T}_{x,y} = \inf\big\{ A^{t,T}(\gamma)\,:\,\gamma\in \Gamma_{x,y}^{t,T}\big\},
\end{equation}
and our main results  concern the asymptotic behavior, as $T\to\infty$, of 
the optimal action among paths with average slope $v\in \R$ on $[0,T]$,
\begin{equation}\label{eq:AF_def}
    A_*^T(v) = A^{0,T}_{0,Tv}.
\end{equation}

\subsection{The Random Potential}
\label{sec:F}
 We consider random potentials that are constructed through a marked Poisson process. Though we give the rigorous details below, one can think of our random field as given by the following sum:
\begin{equation}\label{eq:Fsumrepresentation}
    F(t,x) = \sum_i \varphi_i(t-t_i,x-x_i)
\end{equation}
where $\{(t_i,x_i)\}_{i\in \N}$ are Poisson points on $\R\times \R^d$ and $\{\varphi_i\}_{i\in\N}$ is an i.i.d. family of $C^2$-functions with compact support satisfying certain assumptions. A specific case of our field $F$ is a mollification of a Poisson process on $\R\times \R^d$ (with Lebesgue reference measure) by a fixed $C^2$ function with compact support. In addition, our set-up allows multiplication of the mollifier by an independent random variable satisfying certain moment conditions. 
Further, we allow the mollifiers themselves to be drawn independently at each Poisson point. We will now detail the rigorous set-up that includes the preceding examples.

Fix a compact set $K\subset \R^{d+1}$ with non-empty interior, and let $\mathscr{C}$ be the space of continuous functions $f:\R\times \R^d\to \R$ that are twice differentiable in the spatial variable $x\in \R^d$ with continuous first and second spatial derivatives on $\R\times \R^d,$ and whose support is contained in $K$. The space $\mathscr{C}$ is a separable Banach space when endowed with the norm
\[\|f\|_{\mathscr{C}} = \|f\|_{L^\infty(K)} + \|\nabla f\|_{L^\infty(K)} + \|\nabla^2 f\|_{L^\infty(K)}.\]

Let $\varphiLaw$ be a probability measure on $\mathscr{C}$. We require 
\begin{equation}\label{eq:varphiMoments}
    \varphiLaw\exp\Big(\lambda\|\varphi_1\|_{\mathscr{C}}\Big) < \infty
\end{equation}
for some $\lambda > 0.$ Then, one can construct the random measure $\mathbf{N}$ defined on a complete probability space $(\Omega,\mathcal{F},\Prb)$ that is a Poisson process on $\R\times \R^d\times \mathscr{C}$ with intensity measure given by $dt\otimes dx\otimes \varphiLaw.$ We will use $\omega$ to denote an element of $\Omega$. If we choose an enumeration $\{(t_i,x_i)\}_{i\in\N}$ of Poissonian points on $\R\times \R^d,$ and take an i.i.d.\ sequence $\{\varphi_i\}_{i\in \N}$ with common distribution $\varphiLaw,$ then we have the alternative representation
\[\mathbf{N} = \sum_{i}\delta_{(t_i,x_i,\varphi_i)}.\]
We refer to \cite{Kingman:MR1207584} or \cite{Daley:MR1950431} for the theory of marked Poisson processes.

We define our field $F$ to be
\[F(t,x) = \int_{\R\times \R^d\times \mathscr{C}}\varphi(t-s,x-y)\mathbf{N}(ds,dy,d\varphi).\]
Often we will suppress the domain of integration when integrating with respect to~$\mathbf{N}$. Alternatively, we can represent $F$ as a sum over Poisson points on $\R\times \R^d$ mollified with an i.i.d. sequence of functions $\{\varphi_i\}_{i\in\N}$ with distribution $\varphiLaw$ as in \eqref{eq:Fsumrepresentation}.


\subsection{The generalized kinetic energy}
\label{sec:reqs_on_L}

We assume that $L:\R^d\to \R$ is convex and twice differentiable. Also, we assume
\begin{equation}\label{eq:superLinearGrowthAssumption}
    \lim_{\|v\|_2\to \infty}\frac{L(v)}{\|v\|_2} = +\infty
\end{equation}
and that there is a $\delta_0 \in (0,1)$ such that 
\begin{equation}\label{eq:secondDerivAssumption}
    \limsup_{\|v\|_2\to \infty}\sup_{\|r\|_2\le \delta_0}\frac{\|\nabla^2 L(v+r)\|}{L(v)} < \infty.
\end{equation}
Convexity of $L$ and \eqref{eq:superLinearGrowthAssumption} are needed to guarantee existence of minimizers in \eqref{eq:p2pAction}. The twice differentiable assumption is essential as we will use a second order Taylor approximation for $L$ in the proof of differentiability of the shape function. Condition~\eqref{eq:secondDerivAssumption} is a technical condition limiting the growth and oscillation of $L$ at infinity. We use \eqref{eq:secondDerivAssumption} to control the second derivative in a second order Taylor expansion.

For a concrete example, Lagrangians of the form
\[
L(v)=\sum_{k=2}^{p}a_k \|v\|_2^{p}
\]
with $p\in\N$, $a_p>0$, and  $a_0,a_1,\ldots, a_{p-1}\ge 0$, satisfy all our assumptions.

\subsection{Results}
We recall the definition \eqref{eq:p2pAction}--\eqref{eq:AF_def} and begin with a shape theorem for our model. 

\begin{theorem}\label{thm:shapeTheorem}
    There is a convex deterministic function $\Lambda:\R^d\to \R$ such that for every $v\in \R^d,$
    \begin{equation}\label{eq:shapeTheoremLimit}
        \Lambda(v) = \lim_{T\to \infty}\frac{1}{T}A_*^T(v)
    \end{equation}
    $\Prb$-almost surely.
\end{theorem}
We sketch a fairly standard proof of this result 
in Section~\ref{sec:ShapeTheorem}.

\medskip

We will prove, see \Cref{lem:measurableSelection}, that there is a path $\gamma^T(v)$ (which can be chosen in a measurable way) realizing the infimum in \eqref{eq:p2pAction}--\eqref{eq:AF_def}. We will need this path in the statement of our main result on the differentiability of the shape function $\Lambda$ and a formula for the gradient of $\Lambda$.
Our formula for the gradient
of $\Lambda$ involves an auxiliary field $\Fprime:\R\times \R^d\to \R^d$ defined by
\begin{align}\label{eq:phiDef}
    \Fprime(t,x) &= \int (t-s)\nabla \varphi(t-s,x-y)\mathbf{N}(ds,dy,d\varphi)\\
    \notag
    &= \partial_v \int \varphi(t-s,x-y+v(t-s))\mathbf{N}(ds,dy,d\varphi)\Big|_{v=0}.
\end{align}

\begin{theorem}\label{thm:mainDiffTheorem}
    The function $\Lambda$ is differentiable on all of $\R^d$. In addition, for every $v\in \R^d,$ 
    \begin{equation}\label{eq:shapeFcnDerivative}
        \nabla \Lambda(v) = \lim_{T\to \infty}\frac{1}{T}\int_0^T \Big[\nabla L(\dot\gamma_t^{T}(v)) + \Fprime(t,\gamma^{T}_t(v))\Big]dt
    \end{equation}
    $\Prb$-almost surely.
\end{theorem}
We hope that this formula for $\nabla \Lambda$ may be useful in studying further properties of the effective Lagrangian such as strict convexity, which is conjectured to hold for a broad class of kinetic energies and potentials.

We give a proof of \Cref{thm:mainDiffTheorem} in \Cref{sec:shearedEnvironment,sec:mainTheoremProof}.

The function $A:(0,\infty)\times\R^d\to \R$ 
given by
\[
A(t,x)=A^{0,t}_{0,x},\quad (t,x)\in (0,\infty)\times\R^d,
\]
where the right-hand side is defined in~\eqref{eq:p2pAction},
is a solution of the HJB equation
\begin{equation}\label{eq:HJBDef}
        \partial_t A(t,x) + H(\nabla A(t,x)) = F(t,x),\quad t\in (0,\infty),\ x\in \R^d,
\end{equation}
where the Hamiltonian $H$ is the Legendre--Fenchel transform of $L$:
\[H(p) = \sup_{x\in \R^d}\{\langle p,x\rangle - L(x)\},\quad p\in\R^d.\] 
 The function $A$ satisfies the boundary condition 
\begin{equation}\label{eq:U0Def}
    \lim_{t\searrow 0}A(t,x)=
\begin{cases}
0,& x=0,\\
+\infty,&x\ne 0.
\end{cases}
\end{equation}
It is sometimes called the fundamental solution, see for example the discussion in \cite{Jing-Souganidis-Tran:MR3602941}.

Our main result implies a homogenization theorem with an additional regularity property:
\begin{theorem} \label{th:main-HJB}
For every $t\in(0,\infty)$ and $x\in\R^d$, 
    \begin{equation}\label{eq:pontwiseHomogenization}
        \lim_{\epsilon\searrow 0}\epsilon A(t/\epsilon,x/\epsilon) = t\Lambda(x/t)
    \end{equation}
    $\Prb$-almost surely. The nonrandom function $\overline{U}(t,x) = t\Lambda(x/t)$ is the fundamental viscosity solution of the deterministic HJB equation
    \begin{equation}\label{eq:HJBHomogenized}
        \partial_t \overline U(t,x) + \overline{H}(\nabla \overline{U}(t,x)) = 0,
    \end{equation}
    where $\overline{H}$ is the Legendre--Fenchel transform of $\Lambda$:
\[\overline{H}(p)  = \sup_{\xi\in \R^d}\{\langle \xi,p\rangle - \Lambda(\xi)\},\quad p\in\R^d.\]    
     Moreover, $\overline{H}$ is strictly convex and $\overline{U}(t,x)$ is a \textbf{classical}  solution which is $C^1$ 
    for all $t > 0$, $x\in\R^d$.
\end{theorem}

Stronger than pointwise convergence statements are possible in Theorem~\ref{th:main-HJB} but we do not pursue that direction in this paper concentrating on our new differentiability result.

We give the full proof of \Cref{th:main-HJB}
in Section~\ref{sec:proof-homog}.

The remainder of the paper is structured as follows. The proof of \Cref{thm:mainDiffTheorem} is presented in \Cref{sec:mainTheoremProof}. 
The main idea is introduced in Section~\ref{sec:main_idea} but first a key tool, the sheared environment, is introduced in~\Cref{sec:shearedEnvironment}. In \Cref{sec:convexityLemma} we give a convex analysis lemma, another key ingredient. In \Cref{sec:mainTheoremProofSub} we complete the proof of \Cref{thm:mainDiffTheorem}. \Cref{sec:concentratioBounds,sec:ShapeTheorem,sec:appendix} establish \Cref{thm:shapeTheorem} and prove some more technical results used in the proof of \Cref{thm:mainDiffTheorem}. Specifically, in \Cref{sec:concentratioBounds} we prove certain bounds on the optimal paths and their actions. In \Cref{sec:ShapeTheorem} we present the proof of \Cref{thm:shapeTheorem}.
In \Cref{sec:proof-homog} we prove Theorem~\ref{th:main-HJB}.
 In \Cref{sec:appendix} we give the proof of some technical lemmas.

\section{Proof of Main Theorem}\label{sec:mainTheoremProof}
\subsection{The Sheared Environment}\label{sec:shearedEnvironment}
Working in the sheared environment is a crucial tool in our proof of \Cref{thm:mainDiffTheorem}. Let us introduce the relevant notions first, and then, in Section~\ref{sec:main_idea}, 
we will discuss the main idea of the proof and how the shear transformations are used. 
 
For every $v\in\R^d$, we denote by $\Xi_v$ the shear transformation of 
$\R\times\R^d$ with shear factor~$v$:
\[
\Xi_v(t,x)=(t,x+tv),\quad (t,x)\in\R\times\R^d.
\]
We extend this to paths: for a path $\gamma:I\to\R^d$ defined on $I\subset\R$, we define
$\Xi_v \gamma$ by 
\begin{equation}\label{eq:shearDef}
    (\Xi_v \gamma)_t=\gamma_t+tv,\quad t\in I.
\end{equation}
For any function $f:\R\times\R^d\to\R$, we define
\[
\Xi_v^* f(t,x)= f(t,x+tv),\quad (t,x)\in \R\times\R^d.
\]
If $f$ is vector valued, then we let $\Xi_v^* f$ denote the coordinatewise application of~$\Xi_v^*,$ i.e. $(\Xi_v^* f)_i = \Xi_v^* f_i.$ The shear operators commute with spatial derivatives:
for all $v\in\R$, and all differentiable functions $f$, we have 
$\Xi_v^*\nabla f\equiv \nabla \Xi_v^* f $. 

For a (signed) measure $\nu$ on $\R\times \R^d\times \mathbb{X}$ (for some space $\mathbb{X}$), we denote by $\Xi_v^*\nu$ the pushforward of $\nu$ under~$\Xi_v$ acting on $\R\times \R^d$:
for $A\subset \R\times\R^d$, $B\subset\mathbb{X}$,
\[
\Xi_v^*\nu(A \times B)=\nu(\Xi_v^{-1}(A)\times B).
\]

For a shear factor $v\in \R^d$, we define the \textit{sheared environment}, $F_v,$   by
\begin{align*}
    F_v(t,x)& = \int_{\R\times \R^d\times \mathscr{C}}\varphi(t-s,x+(t-s)v-y)\mathbf{N}(ds,dy,d\varphi)\\
    & = \int_{\R\times \R^d\times \mathscr{C}}\Xi_v^*\varphi(t-s,x-y)\mathbf{N}(ds,dy,d\varphi).
\end{align*}

\begin{lemma}\label{lem:shearActionF}
    We have
    \begin{equation}\label{eq:shearActionF}
        \Xi_v^* F = \int \Xi_v^* \varphi(t-s,x-y)(\Xi_{-v}^* \mathbf{N})(ds,dy,d\varphi).
    \end{equation}
    Consequently,
    \begin{equation}\label{eq:equalityInDistribution}
        F_v \stackrel{d}{=} \Xi_v^* F.
    \end{equation}
\end{lemma}
\begin{proof}
    Applying the change of variables $y=z+sv$ (i.e., the shear transformation~$\Xi_v$), we obtain 
    \begin{align*}
        (\Xi_v^* F)(t,x) &= \int \varphi(t-s,x + tv-y)\mathbf{N}(ds,dy,d\varphi)\\
        & = \int \varphi(t-s,x + (t-s)v + sv - y)\mathbf{N}(ds,dy,d\varphi)\\
        & = \int\varphi(t-s,x + (t-s)v - z)\mathbf{N}(ds,d(z+ sv),d\varphi)\\
        & = \int\Xi_v^*\varphi(t-s,x - z)(\Xi_{-v}^* \mathbf{N})(ds,dz,d\varphi)
    \end{align*}
    proving \eqref{eq:shearActionF}. Equality \eqref{eq:equalityInDistribution} follows from shear invariance of the Poisson process on $\R\times \R^d$.
\end{proof}
For $\gamma \in \Gamma_{0,0}^{0,T}$, we define
\[
B^{T}(v,\gamma) = \int_0^T L(\dot{\gamma}_t + v)dt + \int_0^T F_{v}(t,\gamma_t)dt.
\]
For $v\in\R^d$, let
\begin{equation}\label{eq:BF_def}
    B_*^T(v) = \inf\big\{ B^T(v,\gamma):\,\gamma \in \Gamma_{0,0}^{0,T}\big\}.
\end{equation}

\begin{lemma}\label{lem:measurableSelection}
    For every $x,y\in \R^d$, with $\Prb$-probability one, for every $t,T\in \R$ satisfying $t < T$ there exists a minimizer to \eqref{eq:p2pAction} in $\Gamma_{x,y}^{t,T}$. In addition, there are measurable maps 
    \begin{equation}\label{eq:gammaADef}
        \gamma^T(v):\Omega\to \Gamma^{0,T}_{0,Tv}
    \end{equation}
    and 
    \begin{equation}
        \psi^T(v):\Omega\to \Gamma^{0,T}_{0,0}
    \end{equation}
    providing the infima in \eqref{eq:AF_def} and \eqref{eq:BF_def}, respectively, and such that 
    \begin{align}\label{shearedEnvironmentEquality}
        (B_*^T(v),\Xi_v\psi^T(v))_{T>0} \stackrel{d}{=} (A_*^T(v),\gamma^T(v))_{T>0}.
    \end{align}
\end{lemma}
This lemma justifies the interpretation of $B_*^T(v)$ and $\psi^T(v)$ as the sheared versions of $A_*^T(v)$ and $\gamma^T(v).$ Its proof is mainly technical and so we postpone it to \Cref{sec:appendix}, but \eqref{shearedEnvironmentEquality} is intuitive from Lemma~\ref{lem:shearActionF}. 

An important consequence of this lemma and the Limit Shape Theorem~\ref{thm:shapeTheorem} is that 
for the analysis of the shape function it suffices to consider paths only in $\Gamma^{0,T}_{0,0}$ for a fixed realization of $\mathbf{N}(\cdot,\cdot,\cdot)$, since
$\Prb$-almost surely,
\begin{equation}
\Lambda(v)=\lim_{T\to\infty}\frac{1}{T}B^T_*(v).
\end{equation}
This identity and its generalizations will be used in the proof of the differentiability result in Section~\ref{sec:mainTheoremProofSub}.

\subsection{The main idea of the proof of \Cref{thm:mainDiffTheorem}}
\label{sec:main_idea}
 To prove differentiability of~$\Lambda$ at a point $v$, we want to 
compare $\frac{1}{T}B^T_*(v)$ and $\frac{1}{T}B^T_*(w)$ for $w$ close to $v$, and then sequentially take the limits $T\to \infty$ and $w\to v.$ 

We can use  $\psi^T(v)$ to obtain an upper bound on $B^T_*(w)$:
\[B_*^T(w) \le B^T(w,\psi^T(v)).\]
Using the second order Taylor approximation, we can estimate the right-hand side, $B^T(w,\psi^T(v))$, via  $B^T(v,\psi^T(v)$ and derivatives of $B^T(w,\psi^T(v))$ with respect to $w$.
\begin{equation}
\frac{1}{T}B_*^T(w) \le\frac{1}{T} B^T(v,\psi^T(v)) + \Big\langle w-v, \frac{1}{T}G_T(\psi^T(v)) \Big\rangle + O(\|w-v\|_2^2)
\label{eq:linear-domination}
\end{equation}
where 
\[
G_T(\psi^T(v))=\nabla_w B^T(w,\psi^T(v)) \big|_{w=v}
\] 
allows for an explicit formula, see \eqref{eq:approxLinearDominationB}, and the error term $O(\|w-v\|_2^2)$ is uniform in $T$. The estimate~\eqref{eq:linear-domination} can be used (see \Cref{lem:deterministicDifferentiability}) to find 
$\xi=\lim_{T\to\infty} G_T(\psi^T(v))$ and to prove that the 
subdifferential of the convex function~$\Lambda$ at $v$ is given by a one-point set, $\{\xi\}$, implying that $\xi=\nabla \Lambda(v)$.

We first present the crucial convex analysis lemma in \Cref{sec:convexityLemma}. In \Cref{sec:mainTheoremProofSub} we first present lemmas that allow us to bound the second order term in \eqref{eq:linear-domination}, and then we rigorously implement the preceding argument.

\subsection{Approximate Local Linear Domination}\label{sec:convexityLemma}
First we review some facts concerning convex and concave functions $f:\R^d\to \R$, which can be found, e.g., in \cite{lars-erik_2019}, and we introduce some notation. For a function $f$, we define the subdifferential to be 
\[\partial^\vee f(x) = \{\xi\in \R^d\,:\,\forall y\in \R^d,\,f(y)- f(x)\ge \langle \xi,y-x\rangle\}\]
and the superdifferential to be
\[\partial^\wedge f(x) = \{\xi\in \R^d\,:\,\forall y\in \R^d,\,f(y)- f(x)\le \langle \xi,y-x\rangle\}.\]
If $f$ is convex, then $\partial^\vee f(x)$ is non-empty and convex for all $x\in \R^d.$ A similar statement holds for concave $f$ and $\partial^\wedge f(x).$ A convex function is differentiable at $x\in \R^d$ if and only if $\partial^\vee f(x)$ contains exactly one element.

For $r > 0,$ $x\in \R^d$, we let 
\[B_r(x) = \{y\in \R^d\,:\,\|x-y\|_2<r\}\]
be the ball of radius $r$ around $x$. For $\xi \in \R^d,$ let 
\begin{equation}
    \partial f(x;\xi) = \lim_{r\searrow 0}\frac{f(x + r \xi) - f(x)}{r}
\end{equation}
(assuming the limit exists) be the directional derivative of a function $f$ in direction~$\xi.$ These directional derivatives exist for all $\xi \in \R^d$ and $x$ in the interior of the domain of $f$ if $f$ is either concave or convex.

The following lemma is an arbitrary dimension generalization of the approximate local linear domination result for the one-dimensional case that is Lemma 4.1 in~\cite{BakhtinDow_Differentiability}.

\begin{lemma}\label{lem:deterministicDifferentiability}
    Let $\mathcal{O}\subset \R^d$ be an open set, $x_0\in \mathcal{O}$, and $\mathcal{D}\subset \mathcal{O}$ be dense in $\mathcal{O}.$ Let $(f_n)_{n\in \N}$ be a sequence of functions from $\mathcal{O}$ to $\R$ and $f:\mathcal{O}\to \R$ be a function such that for all $x\in \mathcal{D}\cup\{x_0\}$,
    \[\lim_{n\to \infty}f_n(x) = f(x).\]
    Suppose also that there exists a sequence of vectors $(\xi_n)_{n\in \N}$ in $\R^d$ and a function $h:\mathcal{O}\to \R$ such that the following holds:
    \begin{enumerate}
        \item\label{linearDomination1} There is $\delta > 0$ such that for all $x\in \mathcal{D}\cap B_\delta(x_0)$ and $n\in \N,$ 
        \begin{equation}\label{finiteDerivativeInequality}
            f_n(x) - f_n(x_0) \le \langle \xi_n, x-x_0 \rangle + h(x),
        \end{equation}
        \item\label{linearDomination2} $\lim_{x\to x_0}\frac{h(x)}{\|x-x_0\|_2 } = 0$.
    \end{enumerate}
    If $f$ is convex, then: $f$ is differentiable at $x_0$, the sequence $(\xi_n)_{n\in \N}$ converges, and 
    \begin{equation}\label{convexImpliesDifferentiable}
        \nabla f(x_0) = \lim_{n\to \infty}\xi_n.
    \end{equation}
    If $f$ is concave, then the set of limit points of $(\xi_n)_{n\in \N}$ is contained in $\partial^\wedge f(x_0).$
\end{lemma}
\begin{proof}
    Let $\mathcal{G}\subset \R^d\cup \{\infty\}$ denote the set of limit points of the sequence $(\xi_n)_{n\in \N}$ and let $\xi^*\in\mathcal{G}$ be a limit point of some subsequence $(\xi_{n_k})_{k\in \N}$. Taking $k\to \infty$ in the inequality 
    \[f_{n_k}(x) - f_{n_k}(x_0) \le \langle \xi_{n_k}, x-x_0\rangle + h(x),\]
     we obtain 
    \begin{equation}\label{fUpperDifferential}
        f(x) - f(x_0) \le \langle \xi^*, x-x_0\rangle + h(x)
    \end{equation}
    for all $x\in \mathcal{D}\cap B_\delta(x_0).$
    
    Suppose first that $f$ is convex and let $p\in \partial^\vee f(x_0).$ Then, the inequality
    \[f(x) - f(x_0) \ge \langle p, x-x_0\rangle\]
    along with \eqref{fUpperDifferential} implies 
    \begin{equation}\label{gStarInequality}
        0 \le \langle \xi^* - p, x-x_0\rangle + h(x)
    \end{equation}
    for all $x\in \mathcal{D}\cap B_\delta(x_0).$ Let $\zeta \in \R^d$. Since $\mathcal{D}$ is dense in $\mathcal{O},$ there is a sequence $(x_m(\zeta))_{m\in \N}$ converging to $x_0$ such that 
    \[\lim_{m\to \infty}\frac{x_m(\zeta)-x_0}{\|x_m(\zeta)-x_0\|_2} = \zeta.\]
    Inequality \eqref{gStarInequality} then implies 
    \[0 \le \langle \xi^* - p, \zeta\rangle.\]
    Repeating this procedure with $-\zeta$ gives us 
    \[0 = \langle \xi^*-p, \zeta\rangle.\]
    Since the above holds for all $\zeta\in\R^d$, we have $\xi^* = p.$ But $\xi^*\in \mathcal{G}$ and $p\in \partial^\vee f(x_0)$ were arbitrary, and so, in fact, $\lim_{n\to \infty}\xi_n$ is well-defined and
    \[\mathcal{G} = \partial^\vee f(x_0) = \big\{\lim_{n\to \infty}\xi_n\big\},\]
    and hence $f$ is differentiable at $x_0$ with derivative equal to $\lim_{n\to \infty} \xi_n.$

    Now suppose $f$ is concave. Let $\zeta \in \R^d$ and let $(x_m(\zeta))_{m\in \N}$ denote the same sequence as used previously. Dividing \eqref{fUpperDifferential} by $\|x_m(\zeta) - x_0\|_2$ and taking $m\to \infty$, we obtain 
    \begin{equation}\notag
        \lim_{m\to \infty}\frac{f(x_m(\zeta)) - f(x_0)}{\|x_m(\zeta) - x_0\|_2} \le \langle \xi^*,\zeta\rangle.
    \end{equation}
    The limit on the left-hand side of the above exists by concavity of $f$ and equals $\partial f(x_0;\zeta).$ Indeed, if $L>0$  and $R>0$ are  such that 
    \[|f(x) - f(x_0)| \le L \|x-x_0\|_2\]
    for $\|x-x_0\|<R$ and if $r_m = \|x_m(\zeta) - x_0\|_2$, then
    \begin{align*}
        \lim_{m\to \infty}\frac{f(x_m(\zeta)) - f(x_0)}{\|x_m(\zeta) - x_0\|_2} &= L\lim_{m\to \infty}\frac{\|x_m(\zeta) - x_0 - r_m \zeta\|_2}{r_m} + \lim_{m\to \infty}\frac{f(x_0 + r_m\zeta)-f(x_0)}{r_m}\\
        & = \partial f(x_0;\zeta).
    \end{align*}
    
    So, we obtain
    \[\partial f(x_0;\zeta) \le \langle \xi^*, \zeta\rangle.\]
    Since the above holds for all $\zeta \in \R^d$, we can conclude by (the concave version of) Lemma 3.6.3 in \cite{lars-erik_2019} that $\xi^*\in \partial^\wedge f(x_0).$ 
    Hence, $\mathcal{G}\subset \partial^\wedge f(x_0)$ and the claim for concave $f$ is proved. 
\end{proof}

\subsection{Main Proof}\label{sec:mainTheoremProofSub}

It will be helpful for perturbative analysis to introduce additional parameters in \eqref{eq:BF_def}. For $\alpha,\beta > 0$, define 
\begin{equation*}
    B^T(v,\gamma,\alpha,\beta) = \alpha \int_0^T L(\dot\gamma_t+v)dt + \beta\int_0^T F_v (t,\gamma_t)dt
\end{equation*}
and 
\begin{equation}\label{eq:BFab_def}
    B^T_*(v,\alpha,\beta) = \inf\{B^T(v,\gamma,\alpha,\beta)\,:\,\gamma\in\Gamma^{0,T}_{0,0}\}.
\end{equation}
Since our requirements on the function $L$ and the field $F$ are invariant under linear transformations of $L$ and $F$, \Cref{thm:shapeTheorem} holds not just for $B^T_*(v)=B^T_*(v,1,1)$ but for  
$B^T_*(v,\alpha,\beta)$ with
arbitrary $\alpha,\beta>0.$ We obtain the shape function $\Lambda$ depending on $(v,\alpha,\beta)$ such that for all $v\in \R^d$ and all $\alpha,\beta > 0$,
\begin{equation}\label{eq:def-of-Lambda-alpha-beta}
    \Lambda(v,\alpha,\beta) = \lim_{T\to \infty}\frac{1}{T}B_*^T(v,\alpha,\beta)
\end{equation}
with $\Prb$-probability one. We let $\psi^T(v,\alpha,\beta)$ be the measurable selection of minimizer to \eqref{eq:BFab_def} as given in \Cref{lem:measurableSelection}.

We will use some additional notation. For a path $\gamma\in W^{1,1}([0,T];\R^d)$ define 
\begin{equation}
    \overline{L}_T(v,\gamma) = \frac{1}{T}\int_0^T L(\dot\gamma_t + v)dt
\end{equation}
and 
\begin{equation}
    \overline{F}_T(v,\gamma) = \frac{1}{T}\int_0^T F_v(t,\gamma_t)dt.
\end{equation}

\begin{proposition}\label{prop:concavityOfShapeFcn} (i)
    The function $\Lambda$ is concave in $(\alpha,\beta)$. (ii) For every $\alpha,\beta>0,$ $\Prb$-almost surely, the limit points (as $T\to\infty$) of the
    functions
    \[\Big(\overline{L}_T(v,\psi^T(v,\alpha,\beta))\Big)_{T>0}\textrm{ and }\Big(\overline{F}_T(v,\psi^T(v,\alpha,\beta))\Big)_{T>0}\]
    are contained in $\partial_\alpha^\wedge \Lambda(v,\alpha,\beta)$ and $\partial_\beta^\wedge \Lambda(v,\alpha,\beta),$ respectively.
\end{proposition}
\begin{proof} This is a straightforward extension of the proof of Proposition 4.1 in~\cite{BakhtinDow_Differentiability}.

    Let $\alpha_1,\alpha_2,\beta_1,\beta_2> 0$ and $\mu\in [0,1]$. Let $\alpha = \mu\alpha_1 + (1-\mu)\alpha_2$ and $\beta = \mu\beta_1 + (1-\mu)\beta_2.$ Then,
    \begin{align*}
        B_*^T(v,\alpha,\beta) &= \min\Big\{  \alpha \int_0^T L(\dot\gamma_t+v)dt + \beta  \int_0^T F_v(t,\gamma_t)dt\,:\gamma\in\Gamma^{0,T}_{0,0}\Big\}\\
        & \ge \mu \min\Big\{  \alpha_1 \int_0^T L(\dot\gamma_t+v)dt + \beta_1  \int_0^T F_v(t,\gamma_t)dt\,:\gamma\in\Gamma^{0,T}_{0,0}\Big\} \\
        & \qquad\qquad + (1-\mu) \min\Big\{  \alpha_2 \int_0^T L(\dot\gamma_t+v)dt + \beta_2  \int_0^T F_v(t,\gamma_t)dt\,:\gamma\in\Gamma^{0,T}_{0,0}\Big\}\\
        & = \mu B_*^T(v,\alpha_1,\beta_1) + (1-\mu)B_*^T(v,\alpha_2,\beta_2),
    \end{align*}
    which proves concavity of 
    $B_*^T(v,\alpha,\beta)$ in $\alpha$ and $\beta$.   Now, since $\Lambda(v,\alpha,\beta)$ defined in~\eqref{eq:def-of-Lambda-alpha-beta} is
    a limit of concave functions, part (i) follows.   
    
     Let us now prove part (ii). For $\beta,\alpha,\alpha'> 0$,
    we can use the minimizer  $\psi^T(v,\alpha,\beta)$ realizing $B_*^T(v,\alpha,\beta)$ to estimate  
    $B_*^T(v,\alpha',\beta)$:
    \begin{align*}
        B_*^T(v,\alpha',\beta) &\le  \alpha'\int_0^T L(\dot \psi^T_t(v,\alpha,\beta)+v) + \beta\int_0^T F_v(t,\psi^T_t(v,\alpha,\beta))dt\\
        & = B_*^T(v,\alpha,\beta) + (\alpha'-\alpha)\int_0^T L(\dot \psi^T_t(v,\alpha,\beta)+v)\\
        & = B_*^T(v,\alpha,\beta) + 
        (\alpha'-\alpha) T\overline{L}_T(v,\psi^T(v,\alpha,\beta)).
    \end{align*}
    This estimate allows us to apply \Cref{lem:deterministicDifferentiability}. We can choose $\mathcal{D}$ to be any countable dense subset of $\R^d.$ Fix some $\alpha,\beta>0.$ By~\eqref{eq:def-of-Lambda-alpha-beta}, with $\Prb$-probability one, for all $x\in \mathcal{D}\cup\{\alpha\},$ and all sequences $T_n\nearrow+\infty$,
    \[\lim_{n\to \infty}\frac{1}{T_n}B_*^{T_n}(v,x,\beta) = \Lambda(v,x,\beta).\]
    Thus, with $\Prb$-probability one, the conditions of \Cref{lem:deterministicDifferentiability} are satisfied with 
    \begin{align*}
        x_0 &= \alpha,\\
        f_n(x) &= \frac{1}{T_n}B_*^{T_n}(v,x,\beta),\\
        f(x) &= \Lambda(v,x,\beta),\\
        g_n &= \overline{L}_{T_n}(v,\psi^{T_n}(v,\alpha,\beta)),\\
        h(x) &= 0,
    \end{align*}
    and our claim for~$\overline{L}_T$ follows.
    The claim for~$\overline{F}_T$  follows similarly.
\end{proof}

Define 
\[M_{v,\infty}=\limsup_{T\to \infty}\frac{1}{T}\int_0^T\sup_{\|r\|_2\le \delta_0}\|(\Xi_{v+r}^*\nabla^2 L)(\dot{\psi}^T(v))\|dt,\]
where $\delta_0$ is as in \eqref{eq:secondDerivAssumption}. Also, define
\[N_{v,\infty}=\limsup_{T\to \infty}\frac{1}{T}\int_0^T \int_{\R\times \R^d}(t-s)^2 \sup_{\|r\|_2\le 1}\|(\Xi_{v+r}^*\nabla^2  \varphi)(t-s,\psi^T(v)-y)\| \mathbf{N}(ds,dy,d\varphi)dt.\]
The quantities $M_{v,\infty}$ and $N_{v,\infty}$ will be error terms in a second order Taylor expansion. 
\begin{lemma}\label{lem:secondOrderBounds1}
    For every $v\in \R^d,$ we have
    \begin{equation}\label{MvFinite}
        M_{v,\infty} < \infty
    \end{equation}
    with $\Prb$-probability one.
\end{lemma}
\begin{proof}
    By \eqref{eq:secondDerivAssumption}, there are $R,C_1>0$ such that if $\|x\|_2 > R$, then 
    \[\sup_{\|r\|_2\le \delta_0}\|\nabla^2 L(x+v+r)\| \le C_1 L(x + v).\]
    Let $C_2 = \sup_{\|x\|_2\le R+v+\delta_0} \|\nabla^2 L(x)\|$ and $C_3 = \min(\inf_{x\in \R^d}L(x),0).$ Then,
    \begin{align*}
        \frac{1}{T}\int_0^T\sup_{\|r\|_2\le \delta_0}&\|\nabla^2 L(\dot{\psi}^T(v)+v+r)\|dt \\
        &\le \frac{1}{T}\int_0^T C_2 \1_{\|\dot{\psi}^T(v)\|_2 \le R}dt+ \frac{1}{T}\int_0^T\1_{\|\dot{\psi}^T(v)\|_2 > R} C_1 (L(\dot{\psi}^T(v)+v)-C_3)dt\\
        & \le C_2 + \frac{C_1}{T}\int_0^T L(\dot{\psi}^T(v)+v)dt - C_1C_3.
    \end{align*}
    \Cref{prop:concavityOfShapeFcn} implies that the sequence $\frac{1}{T}\int_0^T L(\dot{\psi}^T(v)+v)dt$ is bounded, and thus $M_{v,\infty}$ is finite.
\end{proof}

\begin{lemma}\label{lem:secondOrderBounds2}
    For every $v\in \R^d,$ we have
    \begin{equation}\label{NvFinite}
        N_{v,\infty} < \infty
    \end{equation}
    with $\Prb$-probability one.
\end{lemma}
We postpone the proof of \Cref{lem:secondOrderBounds2} to \Cref{sec:appendix} because it requires machinery and notation introduced in \Cref{sec:concentratioBounds} and is mainly technical.

We are now prepared to give the proof of  \Cref{thm:mainDiffTheorem}.
\begin{proof}[Proof of \Cref{thm:mainDiffTheorem}]
    For every $x,v,w\in \R^d$, where $v,w$ satisfy $\|v-w\|_2 \le \delta_0,$ Taylor's Theorem implies
    \begin{equation}\label{VTaylorBound}
        L(x + w) \le L(x + v) + \Big\langle w-v, \nabla L(x+v)\Big\rangle + \frac{1}{2}\|w-v\|_2^2\sup_{\|r\|\le \delta_0}\|\nabla^2 L(x+v+r)\|.
    \end{equation}
    Also, if $\|v-w\|_2\le 1$, then for every $s\in \R$ and $\varphi\in \mathscr{C},$
    \begin{align}\label{varphiTaylorBound}
        \Big|\varphi(s,x +  s w) - \varphi(s,x& + sv) - s \Big\langle w-v, \nabla \varphi(s,x + sv)\Big\rangle\Big|\notag\\
        & \le  \frac{1}{2}s^2\|w-v\|_2^2 \sup_{\|r\|_2\le 1}\|\nabla^2 \varphi(s,x+s(v+r))\|\notag\\
        & = \frac{1}{2}s^2\|w-v\|_2^2 \sup_{\|r\|_2\le 1}\|\Xi_{v+r}^*\nabla^2 \varphi(s,x)\|.
    \end{align}
    Using \eqref{varphiTaylorBound}, we obtain
    \begin{align}\label{FwUpperBound}
        F_w(t,x) & = \int \varphi(t-s,x-y+(t-s)w)\mathbf{N}(ds,dy,d\varphi)\notag\\
        &\le F_v(t,x) + \Big\langle w-v , \int (t-s)(\Xi_v^*\nabla \varphi)(t-s,x-y)\mathbf{N}(ds,dy,d\varphi)\Big\rangle\notag\\
        &\qquad + \frac{1}{2}\|w-v\|_2^2 \int(t-s)^2 \sup_{\|r\|_2\le 1}\|( \Xi_{v+r}^*\nabla^2  \varphi)(t-s,x-y)\|\mathbf{N}(ds,dy,d\varphi).
    \end{align}

    Combining \eqref{VTaylorBound} and \eqref{FwUpperBound}, we obtain
    \begin{align}\label{eq:BwvUpperBound}
        B_*^T(w) \le B_*^T(v) + \Big \langle w-v, G_T(\psi^T(v))\Big\rangle + \frac{1}{2}\|w-v\|_2^2 h_T(v)
    \end{align}
    where $G_T$ is defined for $\gamma\in W^{1,1}([0,T];\R^d)$ by
    \[G_T(\gamma) = \int_0^T \nabla L(\dot \gamma_t+v)dt + \int_0^T \int (t-s)(\Xi_v^*\nabla \varphi)(t-s,\gamma_t-y)\mathbf{N}(ds,dy,d\varphi)dt\]
    and $h_T(v)$ is defined as 
    \begin{align*}
        h_T(v) = \int_0^T& \sup_{\|r\|_2\le \delta_0}\|\nabla^2 L(\dot{ \psi}^T(v)+v+r)\|dt \\
        & + \int_0^T \int(t-s)^2 \sup_{\|r\|_2\le \delta_0}\|(\Xi_{v+r}^*\nabla^2  \varphi)(t-s,\psi^T(v)-y)\|\mathbf{N}(ds,dy,d\varphi)dt.
    \end{align*}
    By \Cref{lem:secondOrderBounds1,lem:secondOrderBounds2}, for sufficiently large $T$,
    \[\frac{1}{T}h_T(v) \le M_{v,\infty} + N_{v,\infty} + 1.\]
    So, for sufficiently large $T$, $\frac{1}{T}B_*^T$ satisfies the approximate linear domination requirement in \eqref{finiteDerivativeInequality} at $v$: for any $w$ satisfying $\|w-v\|_2\le \delta_0,$
    \begin{equation}\label{eq:approxLinearDominationB}
        \frac{1}{T}B_*^T(w) \le \frac{1}{T}B_*^T(v) + \Big \langle w-v, \frac{1}{T}G_T(\psi^T(v))\Big\rangle + \frac{1}{2}\|w-v\|_2^2 (M_{v,\infty} + N_{v,\infty} + 1).
    \end{equation}
    Let $\mathcal{D}$ be any countable dense subset of $\R^d.$ With $\Prb$-probability one, for all $w\in \mathcal{D},$ and all sequences $T_n\nearrow+\infty$, $\lim_{n\to \infty}\frac{1}{T_n}B_*^{T_n}(w) = \Lambda(w).$ Also, $\Lambda$ is convex by \Cref{thm:shapeTheorem}. \Cref{lem:deterministicDifferentiability} then implies that $\Lambda$ is differentiable at $v$ and that the limit 
    \[\nabla \Lambda(v) = \lim_{T\to \infty}\frac{1}{T}\int_0^T \Big[\nabla L(\dot {\psi}^T(v)+v) + \int (t-s)(\Xi_v^*\nabla \varphi)(t-s,\psi_t^T(v)-y)\mathbf{N}(ds,dy,d\varphi)\Big]dt\]
    holds $\Prb$-almost surely. Applying the change of variables $\gamma = \Xi_v \psi$ 
to the above and using \eqref{shearedEnvironmentEquality} of \Cref{lem:measurableSelection}, we derive \eqref{eq:shapeFcnDerivative}.
\end{proof}

\section{Bounds on Optimal Paths and Their Actions}\label{sec:concentratioBounds}
In this section, we will prove several results concerning $A_*^T(v)$ and~$\gamma^T(v)$ that will be useful in establishing \Cref{thm:shapeTheorem} and \Cref{lem:secondOrderBounds2}. Our analysis is similar to that in \cite{LaGatta_Wehr2010}. We present the details here because our situation is slightly different. Specifically, in \cite{LaGatta_Wehr2010} the authors consider geodesics in a random Riemannian metric, which can be viewed as a 
space-dependent random Lagrangian (with finite range dependence) but there is no random potential in that setting. We instead take a deterministic Lagrangian $L$ and include a random potential $F$. Our analysis is also similar to the analysis in the proof of Lemma 3 in \cite{HoNe3} and Lemma 4.2 in \cite{BCK:MR3110798}.

Let us first prove that $F$ and its derivatives have exponential moments, allowing us to apply standard concentration of measure techniques.
\begin{lemma}\label{lem:expMomentsF}
    Let $r > 0.$ For some $\lambda > 0,$ for all $x\in \R^d$ and all $t\in \R,$ we have 
    \begin{equation}\label{eq:supExpMoment}
        \E\exp\Big(\lambda \sup_{\|x-y\|_2,|t-s|\le r}| F(s,y)|\Big) < \infty
    \end{equation}
    and
    \begin{equation}\label{eq:derivExpMoment}
        \E\exp\Big(\lambda \sup_{\|x-y\|_2,|t-s|\le r}\int \1_{\|y-y'\|_2,|s-s'|<r}\|\nabla^2 \varphi\|_{L^\infty} \mathbf{N}(ds',dy',d\varphi)\Big) < \infty.
    \end{equation}
\end{lemma}
\begin{proof}
    We will prove \eqref{eq:supExpMoment}, and \eqref{eq:derivExpMoment} can be derived similarly.

    Since the support of $\varphi\in \mathscr{C}$ is uniformly bounded, there is $R>0$ such that
    \begin{align*}
        \sup_{\|x-y\|_2,|t-s|\le r}| F(s,y)| \le \int \1_{\|x-y\|_2,|t-s|\le R}\|\varphi\|_{L^\infty}\mathbf{N}(ds,dy,d\varphi).
    \end{align*}
    Thus, conditioned on there being $Y$ Poisson points $(t_1,x_1)\dots,(t_Y,x_Y)$ in $\R\times \R^d$ satisfying $\|x-x_i\|\le R$ and $|t-t_i|\le R,$ we have 
    \begin{equation}
        \sup_{\|x-y\|_2,|t-s|\le r}| F(s,y)| \le \sum_{i=1}^Y \|\varphi_i\|_{L^\infty},
    \end{equation}
    where $\|\varphi_i\|_{L^\infty}$ are independent and identically distributed random variables. Also, $Y$ is independent of the $\varphi_i$'s and has a Poisson distribution.
 In particular, it has finite exponential moments of all orders. For $A = \{(s,y)\in \R\times \R^d\,:\,\|x-y\|,|t-s|\le R\}$, let $\mathcal{F}(A)$ denote the $\sigma$-algebra generated by the restriction of $\mathbf{N}$ to $A\times \mathscr{C}$. Then,
    recalling~\eqref{eq:varphiMoments}, we obtain
    \begin{align*}
        \E\exp\Big(\lambda \sup_{\|x-y\|_2,|t-s|\le r}| F(s,y)|\Big) &= \E\Big[\E[\exp\Big(\lambda \sup_{\|x-y\|_2,|t-s|\le r}| F(s,y)\Big)\,|\,\mathcal{F}(A)]\Big]\\
        & \le \E\Big[\E[\exp\Big(\lambda \sum_{i=1}^Y \|\varphi_i\|_{L^\infty}\Big)\,|\,\mathcal{F}(A)]\Big]\\
        & \le \E\Big[(\varphiLaw\exp(\lambda \|\varphi_1\|_{L^\infty}))^Y\Big]\\
        & = \E\exp\Big( Y \log \varphiLaw\exp(\lambda \|\varphi_1\|_{L^\infty})\Big)\\
        & < \infty.
    \end{align*}
\end{proof}

Now we will introduce some notation allowing us to study discretizations of continuous paths. For $k\in \Z^{d+1}$, we define the box with corner $k$ to be 
\[I_k = k + [0,1)^{d+1}.\]
Note that the $I_k$ are disjoint and 
\[\R^{d+1} = \bigcup_{k\in \Z^{d+1}}I_k.\]

We will call a path $s:\{1,\dots, m\}\to \Z^{d+1}$ satisfying the condition that $\|s(i+1)-s(i)\|_\infty\le 1$ for all $i\in \{1,\dots, m-1\}$ a \textit{nearest $\ell^\infty$-neighbor} path. Let $\mathcal{S}(n)$ be the set of subsets $A\subset \Z^{d+1}$ satisfying the following: 
\begin{itemize}
    \item $0\in A$;
    \item $A$ has $n$ elements;
    \item the set $A$ is connected in the following sense: for any $x,y\in A$, there is a nearest $\ell^\infty$-neighbor path in $A$ connecting $x$ and $y$.
\end{itemize}
The fact that the size of $\mathcal{S}(n)$ grows at most exponentially follows from the same techniques in the theory of lattice animals, see Lemma 3 of \cite{Klarner:MR214489} for example. We present an easy to derive non-optimal result on exponential growth here for completeness.
\begin{lemma}\label{lem:sizeOfSn}
    The set $\mathcal{S}(n)$ has at most $3^{2(d+1)n}$ elements.
\end{lemma}
\begin{proof}
    It is a fact from graph theory that if $A\in \mathcal{S}(n)$, then there is a nearest $\ell^\infty$-neighbor path $s:\{1,\dots, 2n\}\to A$ whose image is all of $A$. Thus, the number of sets in $\mathcal{S}(n)$ is bounded by the number of such paths. The set of nearest $\ell^\infty$-neighbor paths of length $2n$ is at most $3^{2(d+1)n}$, since each step of such a path has at most $3^{d+1}$ options. 
\end{proof}

The following is a useful lemma that says any subset of $\Z^{d+1}$ can be partitioned into a constant number of subsets with points far apart.
\begin{lemma}\label{lem:constantSizePartition}
    For any $r > 0$ and $d\in \N$, there is $N = N(r)>0$ such that the following holds. For every subset $A\subset \Z^{d+1},$ there is a disjoint partition of $A$, $A_1,\dots, A_N$, such that for all $p=1,\dots, N$ and all $x,y\in A_p$ with $x\neq y,$
    \[\|x-y\|_\infty > r.\]
\end{lemma}
\begin{proof}
    Consider the set $\mathbf{R} := \Z^{d+1}\cap [0,w)^{d+1}$ where $w = \lceil r\rceil+1.$ Every $\ell\in \Z^{d+1}$ can be written uniquely as $\ell = k + w i$ for $k\in \mathbf{R}$ and $i\in \Z^{d+1}.$ Let $N$ be the number of elements in $\mathbf{R}$ and label the points of $\mathbf{R}$ as $k_1,\dots, k_N.$ Let $A_p$ be those points $\ell\in A$ that can be written as $\ell = k_p + w i$ for some (unique) $i\in \Z^{d+1}.$ Also, if $j,\ell\in A_p$ and $j\neq \ell,$ then $j = k_p + w i$ and $\ell = k_p + w i'$ for some $i\neq i'$ in $\Z^{d+1}$ and so 
    \[\|j - \ell\|_\infty = w\|i-i'\|_\infty > r.\]
    The set $\{A_1,\dots, A_N\}$ is the desired partition.
\end{proof}

We consider a family $(Y_k)_{k\in \Z^{d+1}}$ of identically distributed random variables. We will use $\Prb,\E$ to express probabilities and expectations with respect to this family. We assume that $R$, the range of dependence for this model, is finite, i.e., if $A,B\subset \Z^{d+1}$ and $\min_{k\in A,k'\in B}\|k-k'\|_\infty > R$, then the families $(Y_k)_{k\in A}$ and $(Y_k)_{k\in B}$ are independent of each other. We also assume that for some $\lambda_0>0$,
\begin{equation}\label{eq:YexpMoment}
    \E\exp(\lambda_0 |Y_1|) < \infty.
\end{equation}

The next lemma is given as Lemma 2.3 in \cite{LaGatta_Wehr2010}. See also Theorem 1 in~\cite{10.1214/aoap/1177005277} for a similar statement in the i.i.d.\ case with weaker moment assumptions.
\begin{lemma}[Lemma 2.3 in \cite{LaGatta_Wehr2010}]\label{lem:pathWithSmallWeight}
    Let $Y=(Y_k)_{k\in \Z^{d+1}}$ be a family of random variables satisfying the preceeding requirements. There is a deterministic $q>0$ and a random a.s.-finite number $N=N(\omega)$, such that if $n>N$ then 
    \[\max_{\{k_1,\dots,k_n\}\in \mathcal{S}(n)} \sum_{i=1}^n |Y_{k_i}|< qn.\]
\end{lemma}
For $k\in  \Z\times \Z^d,$ define 
\[F^*(k) := \inf_{(s,x)\in I_{k}}F(s,x).\]
\Cref{lem:expMomentsF}
allows us to apply \Cref{lem:pathWithSmallWeight} to $Y_k = F^*(k)$ and derive the following:
\begin{corollary}\label{cor:pathWithSmallWeightApplication1}
    There is $q>0$ and an a.s.-finite random $N = N(\omega)$ such that for all $n>N$ and all $\{k_1,\dots, k_n\}\in \mathcal{S}(n),$
    \begin{equation}
        \sum_{i=1}^n F^*(k_i) > -qn.
    \end{equation}
\end{corollary}

A continuous path $\gamma:[0,T]\to \R^d$ with $\gamma_0 = 0$ can be viewed as a subset of  $\R\times\R^d$. We
 define the connected set $S(\gamma)\subset \Z^{d+1}$ as the set of $k\in \Z^{d+1}$ such that
\[\{(t,\gamma_t)\,:\,t\in [0,T]\}\cap I_k \neq \emptyset.\]
We call $S(\gamma)$ the discretization of $\gamma.$ Since $\gamma$ is continuous, $S(\gamma)\in \mathcal{S}(m)$ for some minimal $m = m(\gamma).$ The number $m(\gamma)$ is the total number of boxes through which $\gamma$ travels. We call~$m(\gamma)$ the \textit{discretized length} of $\gamma.$

The following lemma relates the Euclidean length of the path $\gamma$ to its discretized length.
\begin{lemma}\label{lem:discretizedLength}
    There are $c,C > 0$ such that for all $T>0$ and all paths $\gamma\in W^{1,1}([0,T];\R^d)$ with $\gamma_0 = 0,$
    \begin{equation}\label{discretizedLengthLowerBound}
        \int_0^T\|\dot{\gamma}_t\|_2dt \ge c m(\gamma) - T - C.
    \end{equation}
\end{lemma}
\begin{proof}
    We will prove that 
    \begin{equation}\label{eq:euclideanLength}
        \int_0^T \sqrt{\|\dot\gamma_t\|_2^2 + 1}dt \ge c m(\gamma) - C.
    \end{equation}
    The left-hand side of the above is the Euclidean length of the path $t\mapsto (t,\gamma(t))$. Then, \eqref{discretizedLengthLowerBound} will follow since 
    \[\int_0^T \sqrt{\|\dot\gamma_t\|_2^2 + 1}dt  \le \int_0^T \|\dot\gamma_t\|_2dt + T.\]

    Using \Cref{lem:constantSizePartition}, we choose $N\in \N$ so that any subset of $\Z^{d+1}$ can be partitioned into $N$ disjoint subsets whose elements are greater than distance 1 apart. 

    Let $\gamma\in W^{1,1}([0,T];\R^d)$ and let $\{k_{1},\dots, k_m\} = S(\gamma)$ denote discretization of $\gamma$ with $m=m(\gamma)$ its discretized length. Let 
    \[\{\{k_{i_1^1},\dots, k_{i_{n_1}^1}\},\dots,\{k_{i_1^N},\dots k_{i_{n_N}^N}\}\}\]
    be the disjoint partition of $S(\gamma)$ given in \Cref{lem:constantSizePartition} such that for all $p=1,\dots, N$ and $j\neq j',$
    \[\|k_{i_j^p} - k_{i_{j'}^p}\|_\infty > 1.\]
    Let $p^*$ be such that $n_{p^*} = \max_{p=1,\dots, N} n_p.$ We clearly have 
\begin{equation}
\label{eq:npstar_lowerb}
    n_{p^*} \ge m(\gamma) / N. 
\end{equation}    
    For any $c >0,$ if $C \ge 2Nc$ and $m(\gamma) < 2N$, then inequality  \eqref{eq:euclideanLength} is trivially satisfied. Assume then that $m(\gamma) \ge 2N$. This implies $n_{p^*} \ge 2.$ 
    Let $J_j=I_{k_{i_j^{p^*}}}$. Then, 
    \begin{align}\label{pathLengthLowerBound}
        \int_0^T\sqrt{\|\dot{\gamma}_t\|_2^2 + 1}dt \ge \inf\Big\{\int_0^T&\sqrt{\|\dot{X}_t\|_2^2+1}dt\,:\,\notag\\
        &X\in W^{1,1}([0,T];\R^d)\textrm{ s.t. }\forall i\in \{1,\dots, n_{p^*}\}\,\,X \textrm{ touches } J_i\Big\}.
    \end{align}
    For a path $X$ satisfying the conditions in \eqref{pathLengthLowerBound}, let $t_1,\dots, t_{n_{p^*}}$ in $[0,T]$ be points such that $(t_1,X_{t_1})\in J_1,\dots (t_{n_{p^*}},X_{t_{n_{p^*}}})\in J_{n_{p^*}}$. Note that $\|(t_i,X_{t_i}) - (t_{i'},X_{t_{i'}})\|_\infty > 1$ for all $i\neq i'$. Without loss of generality suppose that $t_1<t_2<\dots<t_{n_{p^*}}$. Since the shortest path between any two points is a straight line, we obtain 
    \begin{align*}
        \int_0^T \sqrt{\|\dot{X}_t\|_2^2+1} dt &\ge \sum_{i=1}^{n_{p^*}-1} \int_{t_i}^{t_{i+1}}\sqrt{\|\dot{X}_t\|_2^2+1} dt \\
        & \ge \sum_{i=1}^{n_{p^*}-1}\|(t_{i+1},X_{t_{i+1}}) - (t_i,X_{t_i})\|_2.
    \end{align*}
    There is a constant $C_1 > 0$ such that $\|x\|_2 \ge C_1 \|x\|_\infty$ for all $x\in \R^{d+1},$ and so we obtain due to~\eqref{eq:npstar_lowerb}
    \begin{equation}\label{eq:euclideanpathLengthLowerBound}
        \int_0^T \sqrt{\|\dot{X}_t\|_2^2+1} dt \ge C_1 (n_{p^*}-1) > C_1 m(\gamma)/N - C_1.
    \end{equation}
Using \eqref{pathLengthLowerBound}, we now complete the proof of  
\eqref{eq:euclideanLength} and the entire lemma.
\end{proof}

Let $\Gamma_{0,\ast}^{0,T} = \bigcup_{y\in \R^d}\Gamma_{0,y}^{0,T}.$ 

\begin{lemma}\label{lem:lowerBoundAction}
    There is $q>0$ such that for every $M>0$ there is $C > 0$  and a $\Prb$-almost surely finite random time $T_0 = T_0(\omega)$ such that if $T>T_0$, then 
    \begin{equation}\label{eq:empiricalEnvironmentLowerBound}
        \frac{1}{T}\int_0^T F(t,\gamma_t)dt \ge -q \frac{m(\gamma)}{T}
    \end{equation}
    and
    \begin{equation}\label{eq:actionLowerBound2}
        \frac{1}{T}A^{0,T}(\gamma) \ge  M \frac{ m(\gamma)}{T} - C
    \end{equation}
    for all $\gamma\in \Gamma_{0,\ast}^{0,T}.$
\end{lemma}
\begin{proof}
    We will work on the almost sure event given in \Cref{cor:pathWithSmallWeightApplication1}. As a consequence, we may assume that $q>0$ and $N=N(\omega)\in \N$ are such that for all $n > N,$ $\{k_1,\dots, k_n\}\in \mathcal{S}$
    \[\sum_{i=1}^n F^*(k_i) > -qn.\]
    Let $T_0(\omega) = N(\omega)$ and take $T>T_0(\omega).$

    Let $\{k_1^*,\dots, k_m^*\}  = S(\gamma)$ be the discretization of $\gamma$, where $m= m(\gamma)$ is the discretized length of $\gamma$. Clearly we must have $m(\gamma) \ge \lfloor T\rfloor.$ Also, observe that the augmented path $t\mapsto (t,\gamma(t))$ spends time at most 1 in each box $I_{k_i^*}$ since the path is directed in time. So,
    \begin{align}\label{eq:potentialLowerBound}
        \frac{1}{T}\int_0^T F(t,\gamma_t)dt & \ge \frac{1}{T}\int_0^T\sum_{i=1}^{m(\gamma)}F^*(k_i^*)\1_{I_{k_i^*}}((\gamma(t),t))dt\notag\\
        &\ge \frac{1}{T}\sum_{i=1}^{m(\gamma)}F^*(k_{i}^*)\notag\\
        & \ge -q \frac{m(\gamma)}{T}.
    \end{align}
    This proves \eqref{eq:empiricalEnvironmentLowerBound}.

    By \eqref{eq:superLinearGrowthAssumption}, for any $r>0,$ there is $C_r>0$ such that 
    \begin{equation}\label{eq:kineticActionLowerBound}
        L(x) \ge r \|x\|_2 - C_r
    \end{equation}
    for all $x\in \R^d.$ Hence \eqref{eq:potentialLowerBound}, \eqref{eq:kineticActionLowerBound}, and \Cref{lem:discretizedLength} imply that for all $\gamma\in\Gamma_{0,*}^{0,T},$ 
    \begin{align*}
        \frac{1}{T}A^{0,T}(\gamma) & = \frac{1}{T}\int_0^T L(\dot{\gamma}_t)dt + \frac{1}{T}\int_0^T F(t,\gamma_t)dt\\
        & \ge \frac{r}{T}\int_0^T\|\dot{\gamma}_t\|_2 dt - C_r - \frac{qm(\gamma)}{T}\\
        & \ge \frac{r c m(\gamma)}{T} - \frac{rC}{T} - r - C_r  - \frac{qm(\gamma)}{T}.
    \end{align*}
    Let $M>0.$ Choose $r > 0$ such that $M = rc - q,$ so that above becomes 
    \begin{equation}\label{linearGrowthFormula}
        \frac{1}{T}A^{0,T}(\gamma) \ge M\frac{m(\gamma)}{T}-\frac{rC}{T} - r - C_r,
    \end{equation}
    and \eqref{eq:actionLowerBound2} follows. 
\end{proof}

\section{Proof of the Shape Theorem}\label{sec:ShapeTheorem}
 Since the method of proof is fairly standard (see for e.g. \cite{10.1007/BFb0074919},\cite{kickb:bakhtin2016}, or \cite{AuffingerDamronHanson_50Years:MR3729447}), we only give a sketch
of the proof of \Cref{thm:shapeTheorem}. See also \cite{Souganidis:MR1697831} or \cite{Rezakhanlou-Tarver:MR1756906} for more PDE based approaches to similar statements.
Fix $v\in \R^d.$ Define for $0\le t < T,$
\[A^{t,T}_*(v) = \inf\{A^{t,T}(\gamma)\,:\,\gamma\in \Gamma_{tv,Tv}^{t,T}\}.\]
Note that $A^T_*(v) = A^{0,T}_*(v)$. For every $0\le t<T,$
\[A_*^{0,T}(v) \le A_*^{0,t}(v)+ A_*^{t,T}(v)\]
since we can concatenate paths realizing the infima on the  right-hand side to estimate the  minimal action on the left-hand side. 
\begin{lemma}\label{lem:integrabilityOfPositiveAction}
    For all $v\in \R,$ $T>0,$
    \[\E\Big[\sup_{t\in [0,T]}\max(A_*^t(v),0)\Big] < \infty.\]
\end{lemma}
\begin{proof} Computing the action of the path $\gamma_s=vs$, we obtain
   \[A_*^t(v) \le t L(v) + \int_0^t F(s,vs)ds \le T \max(L(v),0) + \int_0^T |F(s,vs)|ds.\]
Since $\E[|F(s,vs)|] < \infty$, the result follows.
\end{proof}
The family of maps $(\mathcal{T}_v^t)_{t\ge 0}$ from the space of measures on $\R\times \R^d\times \mathscr{C}$ to itself defined by 
\[\mathcal{T}_v^t(\mathbf{N})(ds,dx,d\varphi) = \mathbf{N}(d(s+t),d(x+tv),d\varphi)\]
is ergodic with respect to the law of $\mathbf{N}$ by properties of the Poisson process. Also, 
\[A^{t,T}_*(v)(\mathbf{N}) =  A^{0,T-t}_*(v)(\mathcal{T}_v^{t}(\mathbf{N})).\]
Kingman's subadditive ergodic theorem then implies the existence of a deterministic $\Lambda(v)\in [-\infty,\infty)$ such that 
\[\Lambda(v) = \lim_{T\to \infty}\frac{1}{T}A^T_*(v)\]
$\Prb$-almost surely. \Cref{lem:lowerBoundAction} implies that
for all $v\in \R^d$
\begin{equation}
    \Lambda(v) > -\infty.
\end{equation}

Now let $v = \alpha v_1 + (1-\alpha )v_2$. We have 
\[\frac{1}{T}A^{0,T}_*(v) = \frac{1}{T}A_{0,Tv}^{0,T} \le \frac{1}{T}A^{0,\alpha T}_{0,\alpha T v_1} + \frac{1}{T}A^{\alpha T,T}_{\alpha T v_1, T v}.\]
Taking $T\to \infty$ and using the convergence in probability of the above terms implies
\[\Lambda(v) \le \alpha \Lambda(v_1) + (1-\alpha)\Lambda(v_2)\]
and thus the convexity of $\Lambda.$ This concludes the proof of \Cref{thm:shapeTheorem}.\epf

\section{Proof of the homogenization result.}
\label{sec:proof-homog}

In this section, we derive \Cref{th:main-HJB} from results given in \Cref{sec:concentratioBounds,sec:ShapeTheorem}.
We begin with the following useful lemma:

\begin{lemma}\label{rem:linearGrowthmGamma}
    For all $\delta>0$ there are $C,C'>0$ such that for all $v\in \R^d,$
    \begin{equation}\label{eq:discretizedLengthUpperBound}
        \limsup_{T\to \infty}\frac{1}{T}m(\gamma^T(v)) \le \delta \Lambda(v) + C
    \end{equation}
    and 
    \begin{equation}\label{eq:actionLowerBound}
        \liminf_{T\to \infty}\frac{1}{T}\int_0^T F(s,\gamma_s^T(v))ds \ge -\delta \Lambda(v) - C'
    \end{equation}
    with $\Prb$-probability one. 
\end{lemma}
\begin{proof}
    Applying \Cref{lem:lowerBoundAction} (with $M=1/\delta$) to $\gamma^T(v)$ and using the convergence of $\frac{1}{T}A^T_*(v)$ to $\Lambda(v)$ as $T\to \infty$ given in \Cref{thm:shapeTheorem} we derive that for any $\delta>0$ there is $C>0$ such that for all $v\in \R^d$ the inequality \eqref{eq:discretizedLengthUpperBound} holds.
    
    Then, \eqref{eq:empiricalEnvironmentLowerBound} of \Cref{lem:lowerBoundAction} implies that there is a $q>0$ such that for all $v\in \R^d,$
    \[\liminf_{T\to \infty}\frac{1}{T}\int_0^T F(s,\gamma_s^T(v;A))ds \ge -q\limsup_{T\to \infty}\frac{1}{T}m(\gamma^T(v)).\]
    Now using~\eqref{eq:discretizedLengthUpperBound} (with $\delta/q$ in place of $\delta$) we obtain \eqref{eq:actionLowerBound}.
\end{proof}

\begin{proof}[Proof of \Cref{th:main-HJB}]
    The limit \eqref{eq:pontwiseHomogenization} is just a restatement of \eqref{eq:shapeTheoremLimit}. Strict convexity of  $\overline H$ is equivalent to differentiability of $\Lambda$ since $\overline H$ and $\Lambda$ are convex duals of each other.
    The verification of \eqref{eq:HJBHomogenized} is done by computing the partial derivatives of $\overline{U}(t,x)=t\Lambda(x/t)$, which exist for all $t>0$ and $x\in \R^d$ due to \Cref{thm:mainDiffTheorem}, and then applying the duality identity
\[\Lambda(\xi) + \overline{H}(\nabla \Lambda(\xi)) = \langle \nabla \Lambda(\xi),\xi\rangle,\]
which holds for all $\xi \in \R^d$ (see Proposition 6.1.1 of \cite{lars-erik_2019}).
    
    The agreement with initial condition
    \begin{equation}
    \label{eq:limit-fundamental} 
    \lim_{t\searrow 0}t\Lambda(x/t)=\begin{cases}
    0,& x=0,\\
    +\infty, &x\ne 0,
    \end{cases}
    \end{equation}
     is trivial for $x=0.$ For $x\neq 0,$ it suffices to prove super-linear growth of $\Lambda$. Since $L$ is convex, we have 
    \begin{align*}
        \epsilon A^{0,t/\epsilon}_{0,x/\epsilon} &= \int_0^{t/\epsilon}L(\dot{\gamma}^{t/\epsilon}_s(x/t))ds +\epsilon \int_0^{t/\epsilon} F(s,\gamma^{t/\epsilon}_s(x/t))ds\notag \\
        &\ge tL(x/t) + \epsilon \int_0^{t/\epsilon} F(s,\gamma^{t/\epsilon}_s(x/t))ds.
    \end{align*}
    Therefore,
    \begin{equation}
    \label{eq:actionLower_corr}
     t \Lambda(x/t) \ge tL(x/t) + t\liminf_{\epsilon\searrow 0}\frac{\epsilon}{t} \int_0^{t/\epsilon} F(s,\gamma^{t/\epsilon}_s(x/t))ds.
    \end{equation}
    Using \eqref{eq:actionLowerBound} of \Cref{rem:linearGrowthmGamma} with $\delta=1$, we obtain that there is $C>0$ such that
    \begin{equation*}
        \liminf_{T\to \infty}\frac{1}{T}\int_0^T F(s,\gamma_s^T(x/t))ds \ge - \Lambda(x/t) - C.
    \end{equation*}
     Applying this with $T=t/\epsilon$  to \eqref{eq:actionLower_corr}, we obtain
        \begin{align*}
        t \Lambda(x/t)  \ge tL(x/t) - t\Lambda(x/t) - Ct.
    \end{align*}
    It follows that 
    \[ t \Lambda(x/t)\ge \frac{1}{2}tL(x/t) - \frac{1}{2}Ct.\]
    For $x\neq 0,$ the right-hand side diverges to $+\infty$ as $t\searrow 0$  by the requirement \eqref{eq:superLinearGrowthAssumption} of super-linear growth of $L$,
    which completes the proof of~\eqref{eq:limit-fundamental} and the entire theorem.
\end{proof}

\section{Technical Lemmas}\label{sec:appendix}
\begin{proof}[Proof of \Cref{lem:secondOrderBounds2}]
    Let $v\in \R^d$ and define
    \[Z_T = \frac{1}{T}\int_0^T \int(t-s)^2 \sup_{\|r\|_2\le 1}\|(\Xi_{v+r}^* \nabla^2 \varphi)(t-s,\psi_t^T(v)-y)\|\mathbf{N}(ds,dy,d\varphi)dt.\]
    Our goal is to prove that $\limsup_{T\to \infty}\frac{1}{T}Z_T<\infty.$

    There is $R > 0$ such that for all $\varphi\in \mathscr{C}$ and all $s,t\in \R$ and $x,y\in \R^d,$
    \begin{equation*}
        \sup_{\|r\|_2\le 1}\|(\Xi_{v+r}^*\nabla^2  \varphi)(t-s,x-y)\| \le \|\nabla^2\varphi\|_{L^\infty} \1_{\|x-y\|_2,|t-s|\le R}.
    \end{equation*}
    Thus, there exists $C_1>0$ such that for all $T>0,$
    \begin{align}
        Z_T & \le C_1 \frac{1}{T}\int_0^T \int \|\nabla^2\varphi\|_{L^\infty} \1_{\|\psi_t^T(v)-y\|_2,|t-s|\le R} \mathbf{N}(ds,dy,d\varphi)dt.\notag
    \end{align}
    Let $\{k_1,\dots,k_m\}\in \mathcal{S}(m)$ denote the discretization of $\psi^T(v)$ (as defined in \Cref{sec:concentratioBounds}). Then, \eqref{eq:crudeUpperBound} implies
    \begin{align}\label{eq:crudeUpperBound}
        Z_T & \le C_1 \frac{1}{T}\sum_{i=1}^{m(\psi^T(v))}\int_0^T \1_{(t,\psi^T_t(v))\in I_{k_i}} \int \|\nabla^2\varphi\|_{L^\infty} \1_{\|\psi_t^T(v)-y\|_2,|t-s|\le R} \mathbf{N}(ds,dy,d\varphi)dt\notag\\
        & \le C_1 \frac{1}{T}\sum_{i=1}^{m(\psi^T(v))}\sup_{(t,x)\in I_{k_i}}\int \|\nabla^2\varphi\|_{L^\infty} \1_{\|x-y\|_2,|t-s|\le R} \mathbf{N}(ds,dy,d\varphi).
    \end{align}
    Now we will apply \Cref{lem:pathWithSmallWeight} to the family $(Y_k)_{k\in\Z^d}$ defined by
    \[Y_k = \sup_{(t,x)\in I_k}\int \|\nabla^2\varphi\|_{L^\infty} \1_{\|x-y\|_2,|t-s|\le R} \mathbf{N}(ds,dy,d\varphi).\]
    Verifying the conditions of \Cref{lem:pathWithSmallWeight} is done similarly to the verification in the proof of \Cref{cor:pathWithSmallWeightApplication1} along with the exponential moment result \eqref{eq:derivExpMoment} of \Cref{lem:expMomentsF}. Using the result of \Cref{lem:pathWithSmallWeight} we derive that there is a $q > 0$ and a random $N=N(\omega)$ such that for all $n> N$ and all $\{k_1,\dots, k_n\}\in \mathcal{S}(n),$ 
    \begin{equation}\label{eq:largenSecondDeriveUpperBound}
        \sum_{i=1}^n\sup_{(t,x)\in I_{k_i}}\int \|\nabla^2\varphi\|_{L^\infty} \1_{\|x-y\|_2,|t-s|\le R} \mathbf{N}(ds,dy,d\varphi) \le qn.
    \end{equation}
    
    Recall that for a continuous path $\gamma\in \Gamma_{0,*}^{0,T},$ $m(\gamma)$ is its discretized length, as introduced in \Cref{sec:concentratioBounds}. For any $v\in \R^d,$ there is $M>0$ such that for any cube $I_k = k + [0,1)^{d+1}$, we can cover the sheared cube $\Xi_v I_k$ by at most $M$ cubes of the form $I_{k'}$ for $k'\in \Z^{d+1}.$ It follows that for any path $\gamma\in\Gamma_{0,*}^{0,T}$,
    \begin{equation}\label{eq:shearDiscretization}
        m(\Xi_v \gamma) \le M m(\gamma).
    \end{equation}    
    
    By \Cref{rem:linearGrowthmGamma}, there is $c > 0$ such that with probability one, $m(\gamma^T(v)) <  cT$ for sufficiently large $T$. By \eqref{eq:shearDiscretization}, there is $C_2 > 0$ such that 
    \[m(\psi^T(v)) \le C_2 T\] 
    for sufficiently large $T$. Also, $m(\psi^T(v))\ge \lfloor T\rfloor$ for all $v\in \R^d,T>0$ so $m(\psi^T(v)) \to \infty$ as $T\to \infty.$ 

    It follows from \eqref{eq:crudeUpperBound}, \eqref{eq:largenSecondDeriveUpperBound}, and the above reasoning that with probability one, for sufficiently large $T$,
    \begin{align*}
        Z_T & \le C_1 \frac{1}{T}\sum_{i=1}^{m(\psi^T(v))}\sup_{(t,x)\in I_{k_i}}\int \|\nabla^2\varphi\|_{L^\infty} \1_{\|x-y\|_2,|t-s|\le R} \mathbf{N}(ds,dy,d\varphi).\\
        & \le  C_1 \frac{ q  m(\psi^T(v))}{T}\\
        & \le  C_1 C_2 q.
    \end{align*}
    We can conclude that \eqref{NvFinite} holds with $\Prb$-probability one.
\end{proof}

In the proof of \Cref{lem:measurableSelection} we will require the following lemma ensuring at most linear growth of the environment $F$. Define the set 
\begin{equation}\label{eq:linearGrowthEvent}
    \Omega_{0} = \Big\{\sup_{x\in \R^d}\sup_{s\in[t,T]}\frac{|F(s,x)|}{\|x\|_2+1} < \infty,\,\,\forall t,T\in \R\text{ satisfying }t <T\Big\}.
\end{equation}
\begin{lemma}\label{lem:linearGrowthPotential}
    The event $\Omega_0$ has $\Prb$-probability one.
\end{lemma}
\begin{proof}
    It suffices to prove that
    \begin{equation}\label{eq:linearGrowth}
        \sup_{x\in \R^d}\sup_{s\in[t,T]}\frac{|F(s,x)|}{\|x\|_2+1} < \infty
    \end{equation}
    holds $\Prb$-almost surely for fixed arbitrary $t < T,$ since the left-hand side of the above is increasing in $T-t.$ Indeed, if \eqref{eq:linearGrowth} holds for some $t < T$, then \eqref{eq:linearGrowth} also holds for $t' < T'$ if $t' > t$ and $T' < T.$ Thus if we have that \eqref{eq:linearGrowth} holds along a countable sequence $t_n,T_n$ with $t_n\to -\infty$ and $T_n\to \infty$ the result follows. Let us now fix some $t,T\in \R$ satisfying $t<T$.

    By \Cref{lem:expMomentsF}, there is a $\lambda > 0$ such that for all $x\in \R,$
    \begin{equation}\label{eq:XkMoments}
        \E\exp\Big(\lambda \sup_{t\le s\le T,\,\|x-y\|_2\le 1} |F(s,y)|\Big) < \infty.
    \end{equation}
    Let $(x_j)_{j\in \N}$ be a sequence of points in $\R^d$ such that 
    \[\R^d = \bigcup_{j\in \N} B_1(x_j)\]
    and 
    \[\limsup_{k\to \infty}\frac{j^{1/d}}{\|x_j\|_2} < \infty.\]

    Let $X_k = \sup_{t\le s\le T,\,y\in B_1(x_j)} |F(s,y)|.$ The sequence $(X_j)_{j\in \N}$ is identically distributed and we have $\E |X_1|^d < \infty$ by \eqref{eq:XkMoments}. Then,
    \begin{align*}
        \E |X_1|^d &= \int_0^\infty \Prb(|X_1|^d > t)dt\\
        & \ge \sum_{j=1}^\infty \Prb(|X_j|^d > j)
    \end{align*}
    and so the right-hand side is finite. By the Borel--Cantelli Lemma, there is a random $C=C(\omega)$ such that $|X_j|\le C(\omega)|j|^{1/d}$ for all $j\in \N.$ Let $C'>0$ be a number such that $|j|^{1/d} \le C'\|x_j\|_2$ for all $j\in \N$ satisfying $\|x_j\|>1.$ Let $x\in \R^d$ be such that $\|x\|_2 > 2.$ We have $x\in B_1(x_{j^*})$ for some $j^*\in \N$ and so
    \[\sup_{t\le s\le T}F(s,x)\le X_{j^*} \le C(\omega)|j^*|^{1/d} \le C(\omega)C' \|x_{j^*}\|_2 \le C(\omega)C'(\|x\|_2 +1).\]
    It follows that \eqref{eq:linearGrowth} holds for arbitrary $t<T$ $\Prb$-almost surely, and hence $\Omega_0$ itself holds $\Prb$-almost surely.
\end{proof}

Let $t<T$ and let $E^{t,T}$ be the subset of $C([t,T]\times \R^d;\R)$ consisting of functions with at most linear growth in space. In other words, $g\in E^{t,T}$ means that $g:[t,T]\times \R^d\to \R$ is continuous, and there are $a,b>0$ (depending on $g$) such that
\[|g(s,x)| \le a\|x\|_2 + b\]
for all $s\in[t,T]$ and $x\in \R^d.$ Note that if $\omega\in \Omega_0,$ then $F$ restricted to $[t,T]\times \R^d$ is in $E^{t,T}$ for all $t<T$. By \Cref{lem:linearGrowthPotential}, $\Omega_0$ has probability one. We can and will assume that $F$ is in $E^{t,T}$ if we neglect a probability zero event. For fixed arbitrary $g\in E^{t,T}$, $x,y\in \R^d$, we can consider the functional $A^{t,T}$ defined in \eqref{eq:AtTDef} to be a function from $\Gamma_{x,y}^{t,T}$ to $\R.$ Recall that unless otherwise specified we endow $\Gamma_{x,y}^{t,T}$ with the $W^{1,1}([t,T];\R^d)$ topology. We can also consider the weak $W^{1,1}$ topology on $\Gamma_{x,y}^{t,T}$, defined in the usual way using the dual space.

In order to prove \Cref{lem:measurableSelection} we will use some results from calculus of variations presented in Chapter 6 of \cite{CannarsaSinestrari_HJB}. In particular, the proof of Theorem 6.1.2 (and the additional comments in Remark 6.2.7) can be used to establish the following lemma.

\begin{lemma}\label{lem:weakContinuityofJ}
    Let $g\in E^{t,T}$ and define the functional
    \[J(\gamma) = \int_t^T L(\dot\gamma_s)ds + \int_t^T g(s,\gamma_s)ds.\]
    For every $x,y\in \R^d,$ the functional $J$ is $W^{1,1}$ weakly sequentially lower semicontinuous on $\Gamma_{x,y}^{t,T}$. In addition, for all $R\ge 0$, the set $\{\gamma\in\Gamma_{x,y}^{t,T}\,:\,J(\gamma)\le R\}$ is sequentially compact in the weak $W^{1,1}$ topology. 
\end{lemma}

\begin{proof}[Proof of \Cref{lem:measurableSelection}]
    Fix $x,y\in \R^d.$ We will prove that a minimizer to \eqref{eq:p2pAction} holds for every $t,T\in \R$ satisfying $t < T$ on the event $\Omega_0$ defined in \eqref{eq:linearGrowthEvent}. By \Cref{lem:linearGrowthPotential} this event has $\Prb$-probability one. Let us now suppose that $\omega\in \Omega_0$, so that the restriction of $F$ to $[t,T]\times \R^d$ is in $E^{t,T}$ and let $t < T$.

    Existence of a minimizer to \eqref{eq:p2pAction} follows directly from Theorem 6.1.2 and Remark 6.2.7 in \cite{CannarsaSinestrari_HJB}, but we will sketch the argument assuming \Cref{lem:weakContinuityofJ}. Take a sequence $(\gamma_n)_{n\in \N}$ in $\Gamma_{x,y}^{t,T}$ such that
    \[\lim_{n\to \infty}A^{t,T}(\gamma_n)= A^{t,T}_{x,y}.\]
    Since for some $R>0,$ $(\gamma_n)_{n\in \N}\subset \{\gamma\in\Gamma_{x,y}^{t,T}\,:\,A^{t,T}(\gamma)\le A_{x,y}^{t,T} + R\}$, \Cref{lem:weakContinuityofJ} implies that there is a $W^{1,1}$ weakly convergent subsequence of $(\gamma_n)_{n\in \N}$ to some $\gamma^*.$ Then, since $A^{t,T}$ is $W^{1,1}$ weakly sequentially lower semicontinuous by \Cref{lem:weakContinuityofJ}, we must have $A^{t,T}(\gamma^*) \le A_{x,y}^{t,T},$ and so in fact $\gamma^*$ is a minimizer to \eqref{eq:p2pAction}.
    


    We now prove that there is a measurable map $\gamma^T(v):\Omega\to \Gamma_{0,Tv}^{0,T}$ that provides the infimum in \eqref{eq:p2pAction}. Consider the set-valued function 
    \begin{align*}
        \Psi:\Omega&\to \mathcal{P}(\Gamma_{0,Tv}^{0,T})\\ 
        \omega&\mapsto \{\gamma\in \Gamma_{0,Tv}^{0,T}\,:\, A^{0,T}(\gamma) = A_*^T(v)\},
    \end{align*}
    where $\mathcal{P}(\Gamma_{0,Tv}^{0,T})$ is the power set of $\Gamma_{0,Tv}^{0,T}.$ 

    We are going to apply the Kuratowski--Ryll-Nardzewski Selection Theorem (see Theorem 18.13 in \cite{Aliprantis:MR2378491}), which will allow 
    us to conclude that there is a measurable map $\gamma^T(v):\Omega\to \Gamma_{0,Tv}^{0,T}$ satisfying $\gamma^T(v)(\omega) \in \Psi(\omega)$, or, equivalently,
    \begin{equation}\label{gammaAnProperty}
        A^{0,T}(\gamma^T(v)) = A_*^{T}(v).
    \end{equation}
    First, note that $\Gamma_{0,Tv}^{0,T}$ endowed with the $W^{1,1}$ topology is a Polish space. We need to check that $\Psi$ is non-empty, takes values in the closed subsets of $\Gamma_{0,Tv}^{0,T},$ and satisfies a set valued measurability condition known as being weakly measurable.

    We have already verified that $\Psi$ is non-empty $\Prb$-almost surely. 

    The fact that $\Psi$ takes values in closed sets of $\Gamma_{0,vT}^{0,T}$ follows from \Cref{lem:weakContinuityofJ}. 

    We now prove that the map $\Psi$ is weakly measurable. Weak measurability means that for every open set $U\subset \Gamma_{0,Tv}^{0,T}$, the set 
    \[U_{\Psi^{-1}} := \{\omega\in \Omega\,:\, \Psi(\omega)\cap U \neq \emptyset\}\]
    is measurable in $\Omega.$ The intuition to proving measurability of $U_{\Psi^{-1}}$ is that we will prove that 1) we can check that a minimizer is in $U$ by checking only countably many elements of $U$ and 2) we can verify that any given path minimizes \eqref{eq:p2pAction} by comparison with a countable subset of $\Gamma_{0,Tv}^{0,T}$. The key will be selecting the right countable subset of $U$ and $\Gamma_{0,Tv}^{0,T}$ that accomplishes these two tasks.

    Let $\mathcal{G}$ be a countable family of continuous functions from $[0,T]\times \R^d$ to $\R$ with compact support that is dense in $C([0,T]\times \R^d;\R)$ in the topology of uniform convergence on compact sets. Let $\mathcal{B}$ be a countable basis of bounded, convex open sets for the $W^{1,1}$ topology on $\Gamma_{0,Tv}^{0,T}$. For an element $B\in \mathcal{B}$ we denote by $\overline{B}$ its closure in $\Gamma_{0,Tv}^{0,T}$. An argument 
similar to our proof of existence of a minimizer to \eqref{eq:AF_def} can be used to show that for all $G\in \mathcal{G}$ and $B\in \mathcal{B}$ there is a minimizer to the minimization problem
    \begin{equation}\label{eq:GBMinimization}
        I_{G,B} := \inf\Big\{\int_0^T V(\dot{\gamma}(t))dt + \int_0^T G(\gamma(t),t)dt\,:\,\gamma\in \overline{B}\Big\}.
    \end{equation}
    Indeed, \Cref{lem:weakContinuityofJ} implies that there is sequence $(\gamma_n)_{n\in \N}$ in $\overline{B}$ such that $A^{0,T}(\gamma_n)\to I_{G,B}$ as $n\to \infty$ and such that $(\gamma_n)_{n\in \N}$ converges weakly in $W^{1,1}$ to some element $\gamma^*\in \Gamma_{0,Tv}^{0,T}$. Since $\overline{B}$ is strongly closed and convex, $\overline{B}$ is closed in the weak $W^{1,1}$ topology. This implies that $\gamma^*\in \overline{B}.$ For $G\in \mathcal{G}$ and $B\in \mathcal{B}$ we denote by $\gamma_{G,B}\in \overline{B}$ some specific choice of minimizer to \eqref{eq:GBMinimization}. Because $\mathcal{B}$ is a basis for the topology on $\Gamma_{0,Tv}^{0,T},$ the family $\Gamma_{\mathcal{G},\mathcal{B}} := \{\gamma_{G,B}\}_{G\in \mathcal{G},B\in \mathcal{B}}$ is dense in $\Gamma_{0,Tv}^{0,T}.$ 

    We let $\mathcal{B}_U$ denote those $B\in \mathcal{B}$ such that $\overline{B}\subset U.$ We claim that 
    \begin{equation}\label{eq:UpsiMeasurable}
        U_{\Psi^{-1}}= \bigcup_{B\in\mathcal{B}_U}\bigcap_{k\in\N}\bigcup_{G\in \mathcal{G}}\bigcap_{\gamma'\in \Gamma_{\mathcal{G},\mathcal{B}} } C(k,\gamma_{G,B},\gamma')
    \end{equation}
    where 
    \[C(k,\gamma,\gamma') = \{\omega\in \Omega\,:\, A^{0,T}(\gamma) \le A^{0,T}(\gamma') + 1/k\}.\]
    For every $\gamma\in \Gamma_{0,Tv}^{0,T}$, the map $\omega\mapsto \int_0^T F(t,\gamma_t)dt$ is clearly measurable. So, for fixed $k,\gamma,\gamma'$, the set $C(k,\gamma,\gamma')$ is measurable. It follows that if \eqref{eq:UpsiMeasurable} holds, then $U_{\Psi^{-1}}$ is measurable and we can conclude that a measurable selection exists. We will now prove that the equality \eqref{eq:UpsiMeasurable} holds.

    First we will prove the forward inclusion of \eqref{eq:UpsiMeasurable}. Let $\omega$ be in $U_{\Psi^{-1}}$. By definition there is $\gamma^*\in U$ that minimizes \eqref{eq:AF_def}. Since $U$ is open, we can find $B\in \mathcal{B}$ be such that $\gamma^*\in B$ and $\overline{B}\subset U.$ Let $S = \sup_{\gamma\in B}\|\gamma\|_{L^\infty}.$ Since $B$ is bounded in $W^{1,1},$ $S$ is finite. For $k\in \N$, let $G_k\in \mathcal{G}$ be such that 
    \[\sup_{\|x\|_{2}  \le S,\,t\in [0,T]}|G_k(t,x) - F(t,x)| < \frac{1}{2T k}.\] 
    Then, by the minimality conditions of $\gamma_{G_k,B}$ and $\gamma^*$ and their bound in $L^\infty$,
    \begin{align*}
        A^{0,T}(\gamma_{G_k,B}) & = \int_0^T L(\dot \gamma_{G_k,B}(t))dt + \int_0^T F(t,\gamma_{G_k,B}(t))dt\\
        &\le \int_0^T L(\dot\gamma_{G_k,B}(t))dt + \int_0^T G_k(t,\gamma_{G_k,B}(t))dt + \int_0^T |G_k(t,\gamma_{G_k,B}(t)) - F(t,\gamma_{G_k,B}(t))|dt\\
        & \le \int_0^T L(\dot\gamma^*(t))dt + \int_0^T G_k(t,\gamma^*(t))dt + \frac{1}{2k}\\
        & \le \int_0^T L(\dot\gamma^*(t))dt + \int_0^T F(t,\gamma^*(t))dt + \frac{1}{k}\\
        & = A^T_*(v) + \frac{1}{k}\\
        & \le A^{0,T}(\gamma') + \frac{1}{k}
    \end{align*}
    for any $\gamma'\in \Gamma_{0,Tv}^{0,T}.$ To summarize, if $\omega\in U_{\Psi^{-1}}$, then there exists $B\in \mathcal{B}_U$ such that for all $k\in \N$ there is $G_k\in \mathcal{G}$ such that for all $\gamma'\in \Gamma_{\mathcal{G},\mathcal{B}}$ we have $A^{0,T}(\gamma_{G_k,B})\le A^{0,T}(\gamma') +1/k$. It follows that 
    \[U_{\Psi^{-1}}\subset \bigcup_{B\in \mathcal{B}_U}\bigcap_{k\in\N}\bigcup_{G\in \mathcal{G}}\bigcap_{\gamma'\in \Gamma_{\mathcal{G},\mathcal{B}} }C(k,\gamma_{G,B},\gamma').\]
    Also, note that we have shown that the sequence $(\gamma_{G_k,B})_{k\in \N}$ in $\Gamma_{\mathcal{G},\mathcal{B}}$ is such that $A^{0,T}(\gamma_{G_k,B})\to A^{T}_*(v)$ as $k\to\infty$.

    Now suppose that 
    \[\omega \in \bigcup_{B\in \mathcal{B}_U}\bigcap_{k\in\N}\bigcup_{G\in \mathcal{G}}\bigcap_{\gamma'\in \Gamma_{\mathcal{G},\mathcal{B}} }C(k,\gamma_{G,B},\gamma').\]
    Let $B\in \mathcal{B}_U$ and $(G_k)_{k\in \N}\subset \mathcal{G}$ be such that 
    \begin{equation}\label{eq:weakMinimizer}
        A^{0,T}(\gamma_{G_k,B}) \le A^{0,T}(\gamma') + 1/k
    \end{equation}
    for all $k\in \N$ and all $\gamma'\in \Gamma_{\mathcal{G},\mathcal{B}}.$ By \Cref{lem:weakContinuityofJ}, there is a subsequence of $(\gamma_{G_k,B})_{k\in \N}$ that converges in the weak $W^{1,1}$ topology to some $\gamma^* \in \Gamma_{0,Tv}^{0,T}.$ But since strongly close convex subsets are weakly closed, $\gamma^*\in \overline{B}\subset U.$ Furthermore, by $W^{1,1}$ weak sequential lower semicontinuity of $A^{0,T}$ and \eqref{eq:weakMinimizer}, we must have 
    \[A^{0,T}(\gamma^*) \le A^{0,T}(\gamma')\]
    for all $\gamma'\in \Gamma_{\mathcal{G},\mathcal{B}}.$ As remarked earlier, there is always a sequence in $\Gamma_{\mathcal{G},\mathcal{B}}$ whose action converges to the minimal action over $\Gamma_{0,Tv}^{0,T},$ and so $\gamma^*$ must in fact be a minimizer over $\Gamma_{0,Tv}^{0,T}$ and hence $\omega \in U_{\Psi^{-1}}.$ We have now shown \eqref{eq:UpsiMeasurable}, and hence $\Psi$ is weakly measurable. Existence of a measurable selection $\gamma^T(v)$ then follows.

    Recall the map $\Xi_v^*$ defined in \eqref{eq:shearDef}. We have by definition of $B_*^T(v)$ and \Cref{lem:shearActionF},
    \begin{equation}\label{BnTransformation}
        B_*^n(v)(\Xi_{-v}^*\mathbf{N}) = A_*^n(v)( \mathbf{N}).
    \end{equation}
    Define
    \[\psi^n(v)(\mathbf{N}) = \Xi_{-v}\gamma^n(v)(\Xi_{-v}^* \mathbf{N}).\]
    Equality \eqref{BnTransformation} and the minimization property of $\gamma^T(v)$ imply that 
    \[B^n(v,\psi^n(v)) = B_*^n(v).\]
    Equality \eqref{shearedEnvironmentEquality} follows by choice of $\psi^n(v)$ and the fact that $\Xi_{-v}^*\mathbf{N}  \stackrel{d}{=} \mathbf{N}.$
\end{proof}

\bibliographystyle{alpha} 
\bibliography{Burgers,polymer}

\begin{thebibliography}{JRAS22b}

\bibitem[AB06]{Aliprantis:MR2378491}
Charalambos~D. Aliprantis and Kim~C. Border.
\newblock {\em Infinite dimensional analysis: a hitchhiker's guide}.
\newblock Springer, Berlin, third edition, 2006.

\bibitem[AD95]{Aldous-Diaconis:MR1355056}
D.~Aldous and P.~Diaconis.
\newblock Hammersley's interacting particle process and longest increasing
  subsequences.
\newblock {\em Probab. Theory Related Fields}, 103(2):199--213, 1995.

\bibitem[ADH17]{AuffingerDamronHanson_50Years:MR3729447}
Antonio Auffinger, Michael Damron, and Jack Hanson.
\newblock {\em 50 years of first-passage percolation}, volume~68 of {\em
  University Lecture Series}.
\newblock American Mathematical Society, Providence, RI, 2017.

\bibitem[Bak16]{kickb:bakhtin2016}
Yuri Bakhtin.
\newblock Inviscid {B}urgers equation with random kick forcing in noncompact
  setting.
\newblock {\em Electron. J. Probab.}, 21:50 pp., 2016.

\bibitem[Bar01]{Baryshnikov:MR1818248}
Yu. Baryshnikov.
\newblock G{UE}s and queues.
\newblock {\em Probab. Theory Related Fields}, 119(2):256--274, 2001.

\bibitem[BCK14]{BCK:MR3110798}
Yuri Bakhtin, Eric Cator, and Konstantin Khanin.
\newblock Space-time stationary solutions for the {B}urgers equation.
\newblock {\em J. Amer. Math. Soc.}, 27(1):193--238, 2014.

\bibitem[BD23]{BakhtinDow_Differentiability}
Yuri Bakhtin and Douglas Dow.
\newblock Differentiability of the shape function for directed polymers in
  continuous space, 2023.

\bibitem[BK18]{BK18}
Yuri Bakhtin and Konstantin Khanin.
\newblock On global solutions of the random {H}amilton-{J}acobi equations and
  the {KPZ} problem.
\newblock {\em Nonlinearity}, 31(4):R93--R121, 2018.

\bibitem[BL19]{Bakhtin-Li:MR3911894}
Yuri Bakhtin and Liying Li.
\newblock Thermodynamic limit for directed polymers and stationary solutions of
  the {B}urgers equation.
\newblock {\em Comm. Pure Appl. Math.}, 72(3):536--619, 2019.

\bibitem[CGGK93]{10.1214/aoap/1177005277}
J.~Theodore Cox, Alberto Gandolfi, Philip~S. Griffin, and Harry Kesten.
\newblock {Greedy Lattice Animals I: Upper Bounds}.
\newblock {\em The Annals of Applied Probability}, 3(4):1151 -- 1169, 1993.

\bibitem[CP11]{CaPi}
Eric Cator and Leandro~P.R. Pimentel.
\newblock A shape theorem and semi-infinite geodesics for the {H}ammersley
  model with random weights.
\newblock {\em ALEA}, 8:163--175, 2011.

\bibitem[CS04]{CannarsaSinestrari_HJB}
Piermarco Cannarsa and Carlo Sinestrari.
\newblock {\em Semiconcave Functions, Hamilton-Jacobi Equations, and Optimal
  Control}.
\newblock Progress in Nonlinear Differential Equations and Their Applications.
  Birkhäuser Boston, MA, 2004.

\bibitem[DRAS18]{Damron:MR3838442}
Michael Damron, Firas Rassoul-Agha, and Timo Sepp\"{a}l\"{a}inen, editors.
\newblock {\em Random growth models}, volume~75 of {\em Proceedings of Symposia
  in Applied Mathematics}.
\newblock American Mathematical Society, Providence, RI, 2018.
\newblock AMS Short Course, Random Growth Models, January 2--3, 2017, Atlanta,
  Georgia, Lecture notes.

\bibitem[DVJ03]{Daley:MR1950431}
D.~J. Daley and D.~Vere-Jones.
\newblock {\em An introduction to the theory of point processes. {V}ol. {I}}.
\newblock Probability and its Applications (New York). Springer-Verlag, New
  York, second edition, 2003.
\newblock Elementary theory and methods.

\bibitem[GTW01]{Gravner-Tracy-Widom:MR1830441}
Janko Gravner, Craig~A. Tracy, and Harold Widom.
\newblock Limit theorems for height fluctuations in a class of discrete space
  and time growth models.
\newblock {\em J. Statist. Phys.}, 102(5-6):1085--1132, 2001.

\bibitem[Ham72]{Hammersley:MR0405665}
J.~M. Hammersley.
\newblock A few seedlings of research.
\newblock In {\em Proceedings of the {S}ixth {B}erkeley {S}ymposium on
  {M}athematical {S}tatistics and {P}robability ({U}niv. {C}alifornia,
  {B}erkeley, {C}alif., 1970/1971), {V}ol. {I}: {T}heory of statistics}, pages
  345--394. Univ. California Press, Berkeley, Calif., 1972.

\bibitem[HMO02]{Hambly-Martin-O'Connell:MR1935124}
B.~M. Hambly, James~B. Martin, and Neil O'Connell.
\newblock Concentration results for a {B}rownian directed percolation problem.
\newblock {\em Stochastic Process. Appl.}, 102(2):207--220, 2002.

\bibitem[HN97]{HoNe3}
C.~Douglas Howard and Charles~M. Newman.
\newblock Euclidean models of first-passage percolation.
\newblock {\em Probability Theory and Related Fields}, 108:153--170, 1997.
\newblock 10.1007/s004400050105.

\bibitem[JRAS22a]{JRAS:https://doi.org/10.48550/arxiv.2211.06779}
Christopher Janjigian, Firas Rassoul-Agha, and Timo Seppäläinen.
\newblock Ergodicity and synchronization of the {K}ardar-{P}arisi-{Z}hang
  equation, 2022.
\newblock arXiv preprint, https://doi.org/10.48550/arxiv.2211.06779.

\bibitem[JRAS22b]{JRAS-JEMS}
Christopher Janjigian, Firas Rassoul-Agha, and Timo Seppäläinen.
\newblock Geometry of geodesics through {B}usemann measures in directed
  last-passage percolation.
\newblock {\em J. Eur. Math. Soc.}, Online first, 2022.

\bibitem[JST17]{Jing-Souganidis-Tran:MR3602941}
Wenjia Jing, Panagiotis~E. Souganidis, and Hung~V. Tran.
\newblock Stochastic homogenization of viscous superquadratic
  {H}amilton-{J}acobi equations in dynamic random environment.
\newblock {\em Res. Math. Sci.}, 4:Paper No. 6, 20, 2017.

\bibitem[Kes86]{10.1007/BFb0074919}
Harry Kesten.
\newblock Aspects of first passage percolation.
\newblock In P.~L. Hennequin, editor, {\em {\'E}cole d'{\'E}t{\'e} de
  Probabilit{\'e}s de Saint Flour XIV - 1984}, pages 125--264, Berlin,
  Heidelberg, 1986. Springer Berlin Heidelberg.

\bibitem[Kin93]{Kingman:MR1207584}
J.~F.~C. Kingman.
\newblock {\em Poisson processes}, volume~3 of {\em Oxford Studies in
  Probability}.
\newblock The Clarendon Press Oxford University Press, New York, 1993.
\newblock Oxford Science Publications.

\bibitem[Kla67]{Klarner:MR214489}
David~A. Klarner.
\newblock Cell growth problems.
\newblock {\em Canadian J. Math.}, 19:851--863, 1967.

\bibitem[KV08]{Kosygina-Varadhan:MR2400607}
Elena Kosygina and S.~R.~S. Varadhan.
\newblock Homogenization of {H}amilton-{J}acobi-{B}ellman equations with
  respect to time-space shifts in a stationary ergodic medium.
\newblock {\em Comm. Pure Appl. Math.}, 61(6):816--847, 2008.

\bibitem[LW10]{LaGatta_Wehr2010}
T.~LaGatta and J.~Wehr.
\newblock A shape theorem for riemannian first-passage percolation.
\newblock {\em Journal of Mathematical Physics}, 51(5):053502, 2010.

\bibitem[MO07]{Moriarty-O'Connell:MR2343849}
J.~Moriarty and N.~O'Connell.
\newblock On the free energy of a directed polymer in a {B}rownian environment.
\newblock {\em Markov Process. Related Fields}, 13(2):251--266, 2007.

\bibitem[NP18]{lars-erik_2019}
Constantin~P. Niculescu and Lars-Erik Persson.
\newblock {\em Convex functions and their applications}.
\newblock CMS Books in Mathematics/Ouvrages de Math\'{e}matiques de la SMC.
  Springer, Cham, 2018.

\bibitem[Ros81]{Rost:MR635270}
H.~Rost.
\newblock Nonequilibrium behaviour of a many particle process: density profile
  and local equilibria.
\newblock {\em Z. Wahrsch. Verw. Gebiete}, 58(1):41--53, 1981.

\bibitem[RT00]{Rezakhanlou-Tarver:MR1756906}
Fraydoun Rezakhanlou and James~E. Tarver.
\newblock Homogenization for stochastic {H}amilton-{J}acobi equations.
\newblock {\em Arch. Ration. Mech. Anal.}, 151(4):277--309, 2000.

\bibitem[Sch09]{Schwab:MR2514380}
Russell~W. Schwab.
\newblock Stochastic homogenization of {H}amilton-{J}acobi equations in
  stationary ergodic spatio-temporal media.
\newblock {\em Indiana Univ. Math. J.}, 58(2):537--581, 2009.

\bibitem[See21]{Seeger:MR4266240}
Benjamin Seeger.
\newblock Homogenization of a stochastically forced {H}amilton-{J}acobi
  equation.
\newblock {\em Ann. Inst. H. Poincar\'{e} C Anal. Non Lin\'{e}aire},
  38(4):1217--1253, 2021.

\bibitem[Sep12]{Seppalainen:MR2917766}
Timo Sepp\"al\"ainen.
\newblock Scaling for a one-dimensional directed polymer with boundary
  conditions.
\newblock {\em Ann. Probab.}, 40(1):19--73, 2012.

\bibitem[Sou99]{Souganidis:MR1697831}
Panagiotis~E. Souganidis.
\newblock Stochastic homogenization of {H}amilton-{J}acobi equations and some
  applications.
\newblock {\em Asymptot. Anal.}, 20(1):1--11, 1999.

\end{thebibliography}

\end{document}